\documentclass[11pt,a4paper,openany]{amsart}
\usepackage{mathrsfs}

\usepackage{amsmath,amsfonts,ams,verbatim}
\usepackage{latexsym}
\usepackage{amssymb,leftidx}
\usepackage{extarrows}
\usepackage{overpic}
\usepackage{color}
\usepackage{epsfig}
\usepackage{subfigure}

\newcommand\specl{dE_{\mathbf{\sqrt{H}}}(\lambda)}
\newcommand\El{E_{\mathbf{\sqrt{H}}}(\lambda)}
\newcommand\Em{E_{\mathbf{\sqrt{H}}}(\mu)}

\newcommand\RR{\mathbb{R}}
\newcommand\ZZ{\mathbb{Z}}
\newcommand\CC{\mathbb{C}}
\newcommand\NN{\mathbb{N}}
\newcommand\Id{\mathrm{Id}}

\newcommand\MMksc{M^2_{k, \sca}}
\newcommand\MMsc{M^2_{\sca}}

\newcommand\MMkb{M^2_{k, b}}

\newcommand\Omegakbh{\Omega_{k,b}^{1/2}}

\newcommand\diagb{\mathrm{diag}_b}
\newcommand\diagsc{\mathrm{diag}_{sc}}

\newcommand\conic{\mathrm{conic}}
\newcommand\lambdabar{{\overline{\lambda}}}

\makeatletter
\@addtoreset{equation}{section}\makeatother

\numberwithin{equation}{section}
\newtheorem{theorem}{Theorem}[section]
\newtheorem{proposition}[theorem]{Proposition}
\newtheorem{lemma}[theorem]{Lemma}
\newtheorem{corollary}[theorem]{Corollary}
\theoremstyle{remark}\newtheorem{remark}[theorem]{Remark}

\newcommand{\zf}{\mathrm{zf}}
\newcommand{\bfo}{\mathrm{bf}_0}
\newcommand{\rbo}{\mathrm{rb}_0}
\newcommand{\lbo}{\mathrm{lb}_0}
\newcommand{\lb}{\mathrm{lb}}
\newcommand{\rb}{\mathrm{rb}}
\newcommand{\bfa}{\mathrm{bf}}
\newcommand{\bfc}{\mathrm{bf}}
\newcommand{\sca}{\mathrm{sc}}

\newcommand{\mf}{\mathrm{mf}}

\begin{document}
\title[Global-in-time Strichartz estimates]{Global-in-time Strichartz estimates on non-trapping asymptotically conic manifolds}

\author{Andrew Hassell}
\address{Department of Mathematics, Australian National University, Canberra ACT 0200, Australia}
\email{Andrew.Hassell@anu.edu.au}

\author{Junyong Zhang}
\address{Department of Mathematics, Beijing Institute of Technology, Beijing 100081 China,
and Department of Mathematics, Australian National University,
Canberra ACT 0200, Australia} \email{zhang\_junyong@bit.edu.cn;
junyong.zhang@anu.edu.au}


\begin{abstract}
We prove global-in-time Strichartz estimates without loss of
derivatives for the solution of the Schr\"odinger equation on a class
of non-trapping asymptotically conic manifolds. We obtain estimates for the full set of admissible indices, including the endpoint, in both the homogeneous and inhomogeneous cases.
This result improves on
the results by Tao, Wunsch and the first author  in \cite{HTW} and \cite{Miz}, which are local in time, as well as the results of the second author in \cite{Zhang}, which are
global in time but with a loss of angular derivatives. In addition, the endpoint inhomogeneous estimate is a strengthened version of the uniform Sobolev estimate recently proved by  Guillarmou and the first author \cite{GH}.
\end{abstract}

\maketitle

\section{Introduction}\label{sec:intro}

Strichartz estimates are an essential tool for studying the behaviour of solutions to nonlinear Schr\"odinger equations, nonlinear wave equations, and other nonlinear dispersive equations. In particular, global-in-time Strichartz estimates are needed to show global well-posedness and scattering for these equations. The purpose of this article is to prove global-in-time Strichartz estimates
for the Schr\"odinger equation on asymptotically conic nontrapping manifolds.

Let $(M^\circ, g)$ be a Riemannian manifold of dimension $n\geq2$, and let $I \subset \RR$ be a time interval.
Strichartz estimates are a family of dispersive estimates on
solutions $u(t,z)$: $I\times M^\circ\rightarrow \mathbb{C}$ to the Schr\"odinger equation
\begin{equation}
i\partial_tu+\Delta_g u=0, \quad u(0)=u_0(z)
\label{gen-Str}\end{equation}
where $\Delta_g$ denotes the Laplace-Beltrami operator on $(M^\circ,
g)$. The general Strichartz estimates state that
\begin{equation*}
\|u(t,z)\|_{L^q_tL^r_z(I\times M^\circ)}\leq
C\|u_0\|_{H^s(M^\circ)},
\end{equation*}
where $H^s$ denotes the $L^2$-Sobolev space over $M^\circ$, and
$(q, r)$ is an \emph{admissible pair}, i.e.
\begin{equation}\label{1.1}
2\leq q,r\leq\infty, \quad 2/q+n/r=n/2,\quad (q,r,n)\neq(2,\infty,2).
\end{equation}
It is well known that \eqref{gen-Str} holds for $(M^\circ, g) = (\RR^n,  \delta)$ with $s=0$ and $I = \RR$.

In this paper, we continue the investigations carried out in
\cite{HTW1,HTW} concerning Strichartz inequalities on a class of
non-Euclidean spaces, that is, smooth complete noncompact
asymptotically conic Riemannian manifolds $(M^\circ,g)$ which
satisfy a non-trapping condition.  Here, `asymptotically conic' means that $M^\circ$ has an end of the form $(r_0, \infty)_r \times Y$, with metric asymptotic to $dr^2 + r^2 h$ as $r \to \infty$, where $(Y, h)$ is a closed Riemannian manifold of dimension $n-1$ (a more precise definition is given below).
%
In \cite{HTW}, the first author, Tao and Wunsch established the
local in time Strichartz inequalities
\begin{equation}\label{1.2}
\|e^{it\Delta_g}u_0\|_{L^q_tL^r_z([0,1]\times M^\circ)}\leq
C\|u_0\|_{L^2(M^\circ)}.
\end{equation}
In this paper, we establish the same inequality on the full time interval, $t \in \RR$.
To treat an infinite time interval, the method of \cite{HTW} no longer works, and we take a completely new approach in this paper (see Section~\ref{intro:strategy}).
Although phrased in terms of asymptotically conic manifolds we emphasize that our results apply in particular to

\vskip 3pt

\noindent$\bullet$ Schr\"odinger operators $\Delta + V$ on $\RR^n$, with $V$ suitably regular and decaying at infinity;

\vskip 1pt

\noindent$\bullet$ nontrapping metric perturbations of flat Euclidean space, with the perturbation suitably regular and decaying at infinity.

\subsection{Geometric setting}\label{subsec:geomsetting} Let us recall the asymptotically conic geometric
setting, which is the same as in \cite{GHS1,GHS2,HW1,HTW}. Let
$(M^\circ,g)$ be a complete noncompact Riemannian manifold of
dimension $n\geq2$ with one end, diffeomorphic to $(0,\infty)\times
Y$ where $Y$ is a smooth compact connected manifold without
boundary. Moreover, we assume $(M^\circ,g)$ is asymptotically conic
which means that $M^\circ$ can be compactified to a manifold $M$
with boundary $\partial M=Y$ such that the metric $g$ becomes a
scattering metric on $M$. That is, in a
collar neighborhood $[0,\epsilon)_x\times \partial M$ of $\partial M$, $g$   takes
the form
\begin{equation}\label{1.3}
g=\frac{\mathrm{d}x^2}{x^4}+\frac{h(x)}{x^2}=\frac{\mathrm{d}x^2}{x^4}+\frac{\sum
h_{jk}(x,y)dy^jdy^k}{x^2},
\end{equation}
where $x\in C^{\infty}(M)$ is a boundary defining function for
$\partial M$ and $h$ is a smooth family of metrics on $Y$. Here we
use $y=(y_1,\cdots,y_{n-1})$ for local coordinates on $Y=\partial
M$, and the local coordinates $(x,y)$ on $M$ near $\partial M$. Away
from $\partial M$, we use $z=(z_1,\cdots,z_n)$ to denote the local
coordinates.
Moreover if every geodesic
$z(s)$ in $M$ reaches $Y$ as $s\rightarrow\pm\infty$, we say $M$ is
nontrapping. The function $r:=1/x$ near $x=0$ can be thought of as a
``radial" variable near infinity and $y = (y_1, \dots, y_{n-1})$ can be regarded as  $n-1$
``angular" variables. Rewriting \eqref{1.3} using coordinates $(r, y)$, we see that  the metric is asymptotic to the exact
conic metric $dr^2+r^2h(0)$ on $(r_0,\infty)_r\times Y$ as $r\rightarrow\infty$.

The Euclidean space $M^\circ=\mathbb{R}^n$, or any compactly supported perturbation of this metric,  is an example of an
asymptotically conic manifold with $Y$ equal to $\mathbb{S}^{n-1}$ endowed with
the standard metric.

\vspace{0.2cm}

Let $(M^\circ,g)$ be an asymptotically conic manifold. The complex
Hilbert space $L^2(M^\circ)$ is given by the inner product
\begin{equation*}
\langle f_1,
f_2\rangle_{L^2(M^\circ)}=\int_{M^\circ}f_1(z)\overline{f_2(z)}dg(z)
\end{equation*}
where $dg(z)=\sqrt{g}dz$ is the measure  induced by the metric $g$. Let
$\Delta_g=\nabla^*\nabla$ be the Laplace-Beltrami operator on
$M$; our sign convention is that $\Delta_g$ is a positive operator.  Let $V$ be a real potential function on $M$ such that
\begin{equation}\label{1.4}
V\in C^\infty(M),~ V(x,y)=O(x^3) ~\text{as}~x\rightarrow0.
\end{equation}
We assume that $n \geq 3$ and that one of the following two conditions hold: either
\begin{equation}\label{1.6} \mathbf{H} := \Delta_g+V~\text{has no zero eigenvalue or zero-resonance},
\end{equation}
or the stronger condition
\begin{equation}\label{1.6strong} \mathbf{H} := \Delta_g+V~\text{has no nonpositive eigenvalues or zero-resonance}.
\end{equation}
By a zero-resonance we mean a nontrivial solution $u$ to $\mathbf{H}u=0$ such that  $u\rightarrow0$ at infinity. Notice that the second assumption, \eqref{1.6strong}, implies that $\mathbf{H}$ is a nonnegative operator, so that we can define $\sqrt{\mathbf{H}}$. These assumptions allow us to use the results of \cite{GHS1}, \cite{GHS2}.

\subsection{Main results} Now we consider the Schr\"odinger equation
\begin{equation}\label{1.7}
i\partial_tu+\mathbf{H} u=0,  \quad u(0, \cdot)=u_0 \in L^2(M).
\end{equation} The main purpose of this paper is to prove
the following results. Notice that the endpoint estimate ($q = 2$ and $\tilde q = 2$) is included in both cases.

\begin{theorem}[Long-time homogeneous Strichartz estimate]\label{Strichartz} Let $(M^\circ,g)$ be an asymptotically conic non-trapping manifold of dimension
$n\geq3$. Let $\mathbf{H}=\Delta_g+V$ satisfy
\eqref{1.4} and \eqref{1.6strong} and suppose $u$ is the solution
to \eqref{1.7}. Then
\begin{equation}
\|u(t,z)\|_{L^q_tL^r_z(\mathbb{R}\times M^\circ)}\leq
C\|u_0\|_{L^2(M^\circ)},
\label{Str-est}\end{equation}
where the admissible pair $(q,r)\in [2,\infty]^2$ satisfies
\eqref{1.1}.
\end{theorem}

\begin{theorem}[Long-time inhomogeneous Strichartz estimate]\label{Strichartz-inhom}
Let $(M^\circ,g)$ and $\mathbf{H}$ be as in
Theorem~\ref{Strichartz}. Suppose that $u$ solves the inhomogeneous
Schr\"odinger equation with zero initial data
\begin{equation}
i \partial_t u + \mathbf{H} u = F(t, z), \quad u(0, \cdot) = 0. \label{uF}\end{equation}
Then the inhomogeneous Strichartz estimate
\begin{equation}
\|u(t,z)\|_{L^q_tL^r_z(\mathbb{R}\times M^\circ)}\leq C \| F
\|_{L^{\tilde q'}_tL^{\tilde r'}_z(\mathbb{R}\times M^\circ)}
\label{eq:inhom}\end{equation}
holds for admissible pairs $(q,r)$, $(\tilde q, \tilde r)$.
\end{theorem}

\begin{remark} If we make the weaker assumption \eqref{1.6}, then the statements above still hold, provided that $u_0$ and $F(t, \cdot)$ lie in the positive spectral subspace for $\mathbf{H}$, or in other words that $u_0 = 1_{[0, \infty)}(\mathbf{H})(u_0)$, and similarly for $F(t, \cdot)$ for almost every $t$.
\end{remark}

\subsection{Strategy of the proof}\label{intro:strategy}
Our argument here extends to long time and to the endpoint
the Strichartz estimates  in \cite{HTW} where the first author, Tao and
Wunsch constructed a ``local" parametrix for the propagator
$e^{it\mathbf{H}}$ based on the parametrix from \cite{HW1}. In that paper, Schr\"odinger solutions $e^{it\mathbf{H}} u_0$ were obtained by applying the parametrix to $u_0$ and then correcting
this approximate solution using Duhamel's formula, using local smoothing estimates to control the correction term. This approach works well on a finite time interval, but cannot be expected to work on an infinite time interval as the errors accumulate over time: certainly they cannot be expected to decay to zero as $t \to \infty$, as would be required to prove $L^q$ estimates in time on an infinite interval.

The main new idea in the current paper is to express the propagator $e^{it\mathbf{H}}$ \emph{exactly} using the spectral measure $dE_{\sqrt{\mathbf{H}}}(\lambda)$, exploiting the very precise information on the spectral measure for the Laplacian on  asymptotically conic nontrapping manifolds has recently become available from the works \cite{HV1}, \cite{HW}, \cite{GHS1}.

After expressing the propagator in terms of an integral of the multiplier $e^{it\lambda^2}$ against the spectral measure, our  strategy is to use the abstract
Strichartz estimate proved in Keel-Tao \cite{KT}. Thus, with $U(t)$ denoting
the (abstract) propagator, we need to show uniform $L^2\rightarrow L^2$ estimates for $U(t)$, and $L^1\rightarrow L^\infty$ type dispersive estimate on
the $U(t) U(s)^*$ with an bound of the form
$O(|t-s|^{-n/2})$. In
the flat Euclidean setting, the estimates are obvious because of the
explicit formula for the propagator. But in our general setting
it turns out to be more complicated. It follows from \cite{HW1} that
\emph{the propagator $U(t)(z,z')$ fails to satisfy such a dispersive estimate at any
pair of conjugate points} $(z,z') \in M^\circ \times M^\circ$ (i.e. pairs $(z,z')$ where geodesics  emanating from $z$
focus at $z'$). Our geometric assumptions allow conjugate points, so we need to modify the propagator
such that the failure of the dispersive estimate at conjugate points is avoided.

This is possible due to the $TT^*$ nature of the estimates required by the Keel-Tao formalism. Recall that the dispersive estimate required by Keel-Tao is of the form
\begin{equation}
\big\| U(t)  U(s)^* \big\|_{L^1 \to L^\infty} \leq C |t-s|^{-n/2}.
\label{KTdisp}\end{equation}
If $U(t)$ is the propagator $e^{it\mathbf{H}}$ then the operator on the left hand side is $e^{i(t-s) \mathbf{H}}$. However, nothing in the Keel-Tao formalism requires the $U(t)$ to form a group of operators. Hence we are free to break up $e^{it\mathbf{H}} = \sum_j U_j(t)$ and prove the estimate \eqref{KTdisp} for each $U_j$. Our choice of $U_j(t)$ (sketched directly below) means that  $ U_j(t)U_j(s)^*$ is essentially the kernel $e^{i(t-s)\mathbf{H}}$ \emph{localized sufficiently close to the diagonal that we avoid pairs of conjugate points}, and hence can prove the dispersive estimate.

Our method of decomposing $e^{it\mathbf{H}} = \sum_j U_j(t)$ is motivated by a decomposition used in the
 proof in \cite{GHS2} of a \emph{restriction estimate}
for the spectral measure, that is, an estimate of the form
$$
\big\| dE_{\sqrt{\mathbf{H}}}(\lambda) \big\|_{L^p(M^\circ) \to
L^{p'}(M^\circ)} \leq C \lambda^{n(\frac1{p} - \frac1{p'}) - 1},
\quad 1 \leq p \leq \frac{2(n+1)}{n+3}.
$$

In \cite{GHS2}, it was observed that to prove a restriction
estimate for $dE_{\sqrt{\mathbf{H}}}(\lambda)$, it suffices (via a $T T^*$ argument) to prove
the same estimate for the operators $Q_j(\lambda)
dE_{\sqrt{\mathbf{H}}}(\lambda) Q_j(\lambda)^*$, where
$Q_j(\lambda)$ is a partition of the identity operator in
$L^2(M^\circ)$. The operators $Q_j(\lambda)$ used in \cite{GHS2} are
pseudodifferential operators (of a certain specific type) serving
to localize $dE_{\sqrt{\mathbf{H}}}(\lambda)$ in phase space close
to the diagonal. The authors of \cite{GHS2} showed
that the localized operators $Q_j(\lambda) dE_{\sqrt{\mathbf{H}}}(\lambda) Q_j(\lambda)^*$
satisfy kernel estimates analogous to those satisfied by the
spectral measure for $\sqrt{\Delta}$ on flat Euclidean space:
\begin{equation}
\Big| \big( Q_j(\lambda)   dE_{\sqrt{\mathbf{H}}}^{(l)}(\lambda)
Q_j(\lambda) \big) (z,z') \Big| \leq C \lambda^{n-1-l} \big( 1 +
\lambda d(z,z') \big)^{-(n-1)/2 + l} , \quad l \in \NN,
\label{spec-meas-j-1}\end{equation} where
$dE_{\sqrt{\mathbf{H}}}^{(l)}(\lambda)$ is the $l$th derivative in
$\lambda$ of the spectral measure, and $d$ is the Riemannian
distance on $M^\circ$.

The authors of \cite{GHS2} hoped that \eqref{spec-meas-j-1} could be used as a `black box' in applications of their work.
Unfortunately, \eqref{spec-meas-j-1} seems inadequate for our present purposes. This is because, in order to obtain the dispersive estimate, we need to efficiently exploit the
oscillation of the `spectral multiplier' $e^{it\lambda^2}$, and
particularly the discrepancy between the way this function oscillates relative to the oscillations (in $\lambda$) of the Schwartz kernel of the spectral measure. \emph{The second main innovation of this paper is to improve estimate \eqref{spec-meas-j-1} on the localized spectral measure.}
We show

\begin{proposition}
\label{prop:localized spectral measure} Let $(M^\circ,g)$ and
$\mathbf{H}$ be in Theorem \ref{Strichartz}. Then there exists a $\lambda$-dependent  operator partition of unity on
$L^2(M)$
$$
\mathrm{Id}=\sum_{j=1}^{N}Q_j(\lambda),
$$
with $N$ independent of $\lambda$,
such that for each $1 \leq j \leq N$ we can write
\begin{equation}\label{beanQ}\begin{gathered}
(Q_j(\lambda)dE_{\sqrt{\mathrm{H}}}(\lambda)Q_j^*(\lambda))(z,z')=\lambda^{n-1} \Big(  \sum_{\pm} e^{\pm
i\lambda d(z,z')}a_\pm(\lambda,z,z') +  b(\lambda, z, z') \Big),
\end{gathered}\end{equation}
with estimates
\begin{equation}\label{bean}\begin{gathered}
\big|\partial_\lambda^\alpha a_\pm(\lambda,z,z') \big|\leq C_\alpha
\lambda^{-\alpha}(1+\lambda d(z,z'))^{-\frac{n-1}2},
\end{gathered}\end{equation}
\begin{equation}\label{beans}\begin{gathered}
\big| \partial_\lambda^\alpha b(\lambda,z,z') \big|\leq C_{\alpha, M}
\lambda^{-\alpha}(1+\lambda d(z,z'))^{-K} \text{ for any } K.
\end{gathered}\end{equation}
Here $d(\cdot, \cdot)$ is the Riemannian distance on $M^\circ$.

\end{proposition}

\begin{remark} 
The estimates \eqref{bean}, \eqref{beans} are easily seen to imply \eqref{spec-meas-j-1} (using Lemma \ref{ledb} to estimate the $\lambda$-derivatives of the operators $Q_i(\lambda)$). However, \eqref{bean}, \eqref{beans} also capture the oscillatory behaviour of the spectral measure, which is crucial in obtaining sharp dispersive estimates in Section~\ref{sec:dispersive}. 
\end{remark}

We now define  localized (in phase space) propagators
$U_j(t)$ by
\begin{equation}\label{Uiti}
\begin{split}
U_j(t) = \int_0^\infty e^{it\lambda^2} Q_j(\lambda)
dE_{\sqrt{\mathbf{H}}}(\lambda), \quad 1 \leq j \leq N.
\end{split}\end{equation}
Then the operator $U_j(t) U_j(s)^*$
is given, at least formally, by (see Lemma \ref{BB*})
\begin{equation}
U_j(t) U_j(s)^* =  \int e^{i(t-s)\lambda^2}  Q_j(\lambda)
dE_{\sqrt{\mathbf{H}}}(\lambda) Q_j(\lambda)^*.
\label{Uiti2}\end{equation}
However,  there are subtleties involved in spectral integrals such as \eqref{Uiti}, \eqref{Uiti2} containing operator-valued functions. Even to show that
\eqref{Uiti} is well-defined as a bounded operator on $L^2(M^\circ)$ is
nontrivial. \emph{The third main innovation of this paper is to give an effective method for analyzing spectral integrals such as \eqref{Uiti}, \eqref{Uiti2} with operator-valued multipliers.}
We use a dyadic decomposition in $\lambda$ and a
Cotlar-Stein almost orthogonality argument to show the
well-definedness of \eqref{Uiti} and prove a uniform estimate on $\| U_j(t) \|_{L^2 \to L^2}$, as required by the Keel-Tao formalism.

Having made sense of \eqref{Uiti2}, we exploit the oscillations both in the multiplier
$e^{i(t-s)\lambda^2}$ and in the localized spectral measure (as expressed by
\eqref{bean} and \eqref{beans}) to obtain the required dispersive estimate for $U_j(t) U_j(s)^*$. The
homogeneous Strichartz estimate for $e^{it\mathbf{H}}$ then follows by applying Keel-Tao to each $U_j$ and summing over $j$.


 Next we consider the  inhomogeneous Strichartz estimates. As is well-known, the non-endpoint cases of the inhomogeneous estimate follow from the homogeneous estimate and the Christ-Kiselev
lemma. The endpoint inhomogeneous estimate requires an additional argument, and in particular, in this case we require estimates on $U_i(t) U_j(s)^*$ for $i \neq j$. This estimate turns out to be very similar to the uniform Sobolev estimate (on asymptotically conic nontrapping manifolds) of Guillarmou and the first author \cite{GH}.  We use the techniques of that paper, in particular a refined partition of the identity operator. This resemblance to the proof in \cite{GH} is more than formal: as pointed out to us by Thomas Duyckaerts and Colin Guillarmou, the inhomogeneous endpoint Strichartz estimate implies the uniform Sobolev estimate; we sketch this argument in Section~\ref{sec:inhom}. Thus, this part of the paper can be regarded as a time-dependent reformulation of the proof in \cite{GH}, leading to a more general result.

\subsection{Previous literature}Now we review some classical
results about the Strichartz estimates. In the flat Euclidean space,
where $M^\circ=\R^n$ and $g_{jk}=\delta_{jk}$, one can take  $I=\R$;
see Strichartz \cite{Str}, Ginibre and Velo \cite{GV}, Keel and Tao
\cite{KT}, and references therein. The now-classic paper \cite{KT}
by Keel-Tao developed an abstract approach to Strichartz estimates which has
become the standard approach in most subsequent literature,
including this paper. Strichartz estimates for compact metric
perturbations of Euclidean space were proved locally in time by
Staffilani and Tataru \cite{ST}, and subsequently for asymptotically
Euclidean manifolds by Robbiano-Zuily \cite{RZ} and Bouclet-Tzvetkov
\cite{BT1}, and in the asymptotically conic setting by
Hassell-Tao-Wunsch \cite{HTW} and Mizutani \cite{Miz}. In these
works, either the metric is assumed nontrapping, or the theorem
holds outside a compact set. In \cite{BGH} the authors proved that
Strichartz estimates without loss hold on an asymptotically conic
manifold with hyperbolic trapped set. Strichartz estimates have also
been studied on exact cones \cite{Ford} and on asymptotically hyperbolic spaces \cite{Bouclet1}.

Strichartz estimates have also been studied on compact manifolds and
on manifolds with boundary. In the compact case, Strichartz
estimates usually are local in time and with some loss of
derivatives $s$ (i.e.~ the RHS of \eqref{Str-est} has to be replaced
by the $H^s$ norm of $u_0$). Estimates for the standard flat 2-torus
were shown by Bourgain \cite{Bourgain} to hold for any $s>0$.  For
any compact manifold, Burq et al. \cite{BGT} showed that the
estimate holds for $s=\frac1q$ and the loss of derivatives, as well
as the localization in time, is sharp on the sphere. Manifolds with boundary were studied in
\cite{BSS1,BSS2}, \cite{Iva1}, \cite{BFSM}.

Global-in-time Strichartz estimates on asymptotically Euclidean spaces  have been proved in
Bouclet-Tzvetkov \cite{BT2} (but with a low energy cutoff), Metcalfe-Tataru \cite{MT}, Marzuola-Metcalfe-Tataru \cite{MMT} and
Marzuola-Metcalfe-Tataru-Tohaneanu \cite{MMTT}.

As already noted,
Strichartz estimates are an essential tool for studying the behaviour of solutions to nonlinear dispersive equations.  There is a vast literature on this topic, and it is beyond the scope of this introduction to review it, so we refer instead to  Tao's book \cite{Tao}  and the references therein.

\vspace{0.2cm}

\subsection{Organization of this paper}  We review the partition of the
identity and properties of the microlocalize the spectral measure
for low energies in Section~\ref{sec:low} and for high frequency in Section~\ref{sec:high}.
In Section~\ref{sec:localized}, we prove Proposition \ref{prop:localized spectral
measure} based on the properties of the microlocalized spectral
measure. Section~\ref{sec:L2} is devoted to the construction of microlocalized
propagators and the proof of the $L^2$-estimates. The dispersive estimates are proved in Section~\ref{sec:dispersive}.
Finally we prove the homogeneous Strichartz estimates in Section~\ref{sec:hom} and the inhomogeneous Strichartz estimates in Section~\ref{sec:inhom}.

\subsection{Acknowledgements} We thank Colin Guillarmou, Adam Sikora, Jean-Marc Bou\-clet,
Thomas Duyckaerts and Pierre Portal for helpful conversations. This
research was supported by Future Fellowship FT0990895 and Discovery
Grants DP1095448 and DP120102019 from the Australian Research
Council. The second author was supported by Beijing Natural Science
Foundation (1144014) and National Natural Science Foundation of
China (11401024).


\section{Spectral measure and partition of the identity at low energies}\label{sec:low}
The spectral measure for the operator $\mathbf{H}$ for low energies
was constructed in \cite{GH1}, on the `low energy space' $\MMkb$.
Here we recall the low energy space $\MMkb$ and the associated space
$\MMksc$. The latter space is needed in order to define the class of
pseudodifferential operators in which our operator partition
$Q_j(\lambda)$ from Proposition~\ref{prop:localized spectral measure} lies.

\subsection{Low energy space}
The low energy space $M_{k,b}^2$, defined in \cite{GH1} (based
on unpublished work of Melrose-S\'a Barreto) is a blown-up version of\footnote{In \cite{GH1}, the spectral parameter was denoted $k$ rather than $\lambda$, hence the subscript `$k$' in $M^2_{k,b}$.} $[0,\lambda_0]\times M^2$. This space is illustrated in
Figure 1. More precisely, we define the $3$-codimension corner
$C_3=\{0\}\times\partial M\times\partial M$ and the $2$-codimension
submanifolds
$$C_{2,L}=\{0\}\times\partial M\times M,\quad C_{2,R}=\{0\}\times M\times \partial M,
\quad C_{2,C}=[0,1]\times\partial M\times \partial M.$$ Without loss
of generalities, we assume $\lambda_0=1$. The space $M_{k,b}^2$ is
defined by
$$M_{k,b}^2=\big[[0,1]\times M^2; C_3,
C_{2,R},C_{2,L},C_{2,C}\big]$$ with blow-down map $\beta_b:
M_{k,b}^2\rightarrow [0,1]\times M^2$. Here the notation $[X; Y]$, where $X$ is a manifold with corners and $Y$ a p-submanifold\footnote{We say that $Y$ is a p-submanifold of $X$ if, near every point $p \in Y$, there are local coordinates $x_1, \dots, x_l, y_1, \dots, y_{n-l}$, where $x_i \geq 0$, $y_i \in (-\epsilon, \epsilon)$, $p = (0, \dots, 0)$, such that $Y$ is given locally by the vanishing of some subset of these coordinates.} of $X$, indicates that $Y$ is blown up in $X$ in the real sense; as a set, $[X; Y]$ is the disjoint union of $X \setminus Y$ and the inward-pointing spherical normal bundle of $Y$, $SN^+Y$. Moreover, $[X; Y_1, Y_2, \dots ]$ indicates iterated blowup. See \cite[Section 18]{Mel} for further details.

The new boundary
hypersurfaces created by these blowups are labelled by
$$\mathrm{rb}=\text{clos}\beta_b^{-1}([0,1]\times M\times \partial M), \mathrm{lb}=\text{clos}\beta_b^{-1}([0,1]\times \partial M\times M),
\mathrm{zf}=\text{clos}\beta_b^{-1}(\{0\}\times M\times M),$$ the
`b-face' $\mathrm{bf}=\text{clos}\beta_b^{-1}(C_{2,C}\setminus C_3)$
and
$$\mathrm{bf_0}=\beta_b^{-1}(C_3), ~\mathrm{rb_0}=\text{clos}\beta_b^{-1}(C_{2,R}\setminus C_3),
~\mathrm{lb_0}=\text{clos}\beta_b^{-1}(C_{2,L}\setminus C_3).$$ The
closed lifted diagonal is given by
$\mathrm{diag}_b=\text{clos}\beta_b^{-1}([0,1]\times \{(m,m); m\in
M^\circ\})$, and its intersection with the face $\mathrm{bf}$ is
denoted by $\partial_{\bfc} \mathrm{diag}_b$. We remark that
$\mathrm{zf}$ is canonically diffeomorphic to the b-double space
\begin{equation}
M_b^2=[M^2;\partial M\times\partial M],
\label{M2b}\end{equation}
 as is each section $M^2_{k,b}\cap
\{\lambda=\lambda_*\}$ for fixed $0<\lambda_*<1$.\vspace{0.2cm}

We further define the space $\MMksc$ to be the blowup of $\MMkb$ at $\partial_{\bfc} \diagb$. This space is illustrated in Figure~\ref{fig:mmksc}. The sections $\MMksc \cap \{\lambda=\lambda_*\}$ for fixed $0<\lambda_*<1$ are all canonically diffeomorphic to the scattering double space $\MMsc$, which is
the blowup of $M^2_b$ at the boundary of the lifted diagonal:
$$
\MMsc = [M^2_b; \partial \diagb].
$$
To avoid excessive notation we denote the diagonal in $M^2_b$ and in $M^2_{k,b}$ by the same symbol $\diagb$. 
We similarly define $\diagsc$ to be the closure of the interior of $\diagb$
lifted to $\MMsc$ (or $\MMksc$).\vspace{0.2cm}

\subsection{Coordinates}
Let $(x, y) = (x, y_1, \dots, y_{n-1})$ be local coordinates on $M$ near a boundary point, as discussed in Section~\ref{subsec:geomsetting}.
We define functions $x$ and $y$ on $M_{k,b}^2$ by lifting from the
left factor of $M$ (near $\partial M$), and $x',y'$ by lifting from
the right factor of $M$; similarly $z,z'$ (away from $\partial M$).
Let $\rho=x/\lambda, \rho'=x'/\lambda$, and
$\sigma=\rho/\rho'=x/x'$. Then we can use coordinates
$(y,y',\sigma,\rho',\lambda)$ near $\mathrm{bf}$ and away from
$\mathrm{rb}$; while $(y,y',\sigma^{-1},\rho,\lambda)$ near
$\mathrm{bf}$ and away from $\mathrm{lb}$.

Next we consider local coordinates on the scattering double space $\MMsc$. The only difference between this space and $M^2_b$ is at the boundary of the diagonal.
In local coordinates, near
$\partial_{\bfc} \diagb$, a boundary defining function for $\bfc$ is given by $x/\lambda$, and the diagonal is given by $\sigma = 1, y = y'$. Therefore, coordinates on the interior of the new boundary hypersurface, denoted $\sca$, created by this blowup are
$$
\frac{\lambda (\sigma - 1)}{x}, \ \frac{\lambda(y-y')}{x}, \ \lambda, \ y'.
$$

\begin{figure}[ht!]
\begin{center}
\input{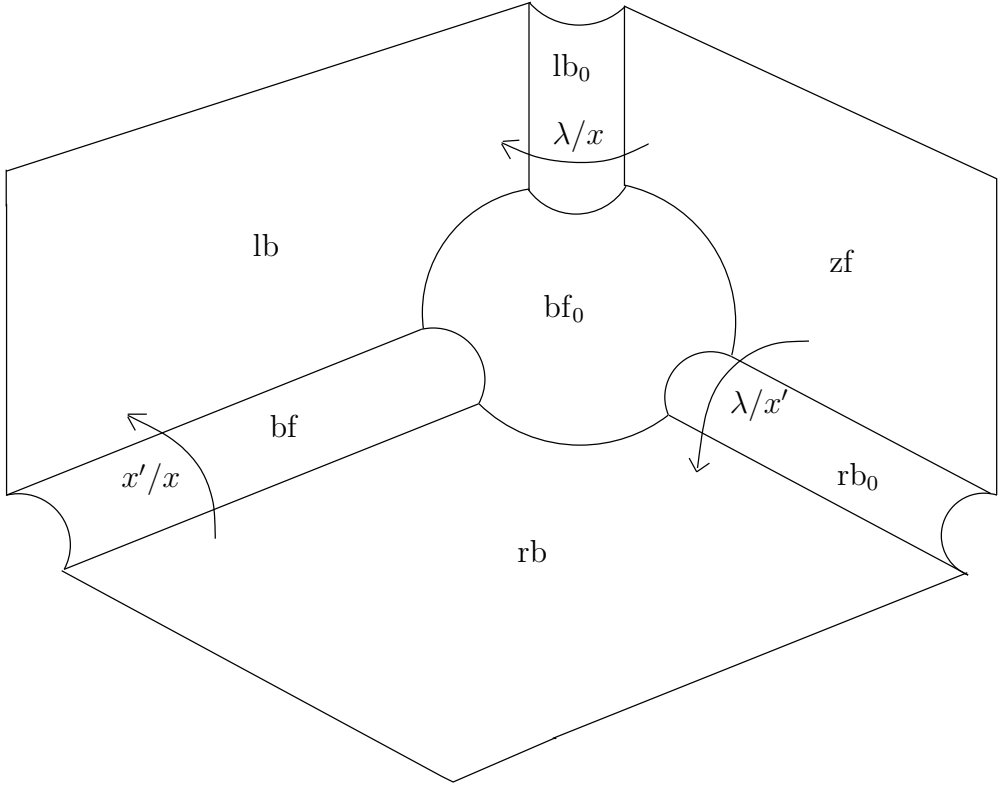}
\caption{The manifold $M^2_{k,b}$. Arrows show the direction in which the indicated function increases from $0$ to $\infty$.}
\label{mmkb}
\end{center}
\end{figure}

\begin{figure}[ht!]
\begin{center}
\input{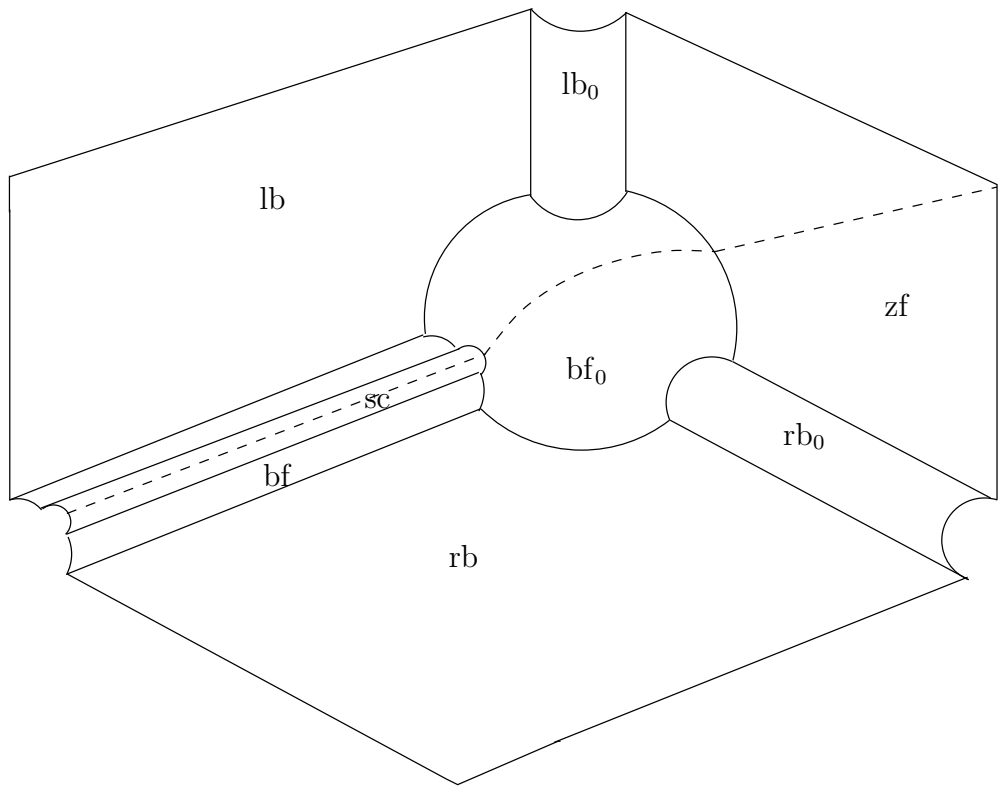}
\caption{The manifold $M^2_{k,\sca}$; the dashed line is the boundary of the lifted diagonal $\Delta_{k, \sca}$}
\label{fig:mmksc}
\end{center}
\end{figure}

We also need to consider coordinates on phase space. As emphasized by Melrose \cite{Mel}, the appropriate phase space for analyzing the Laplacian with respect to a scattering metric is the scattering cotangent bundle. This is the dual space of the scattering tangent bundle ${}^{sc}TM$, which is the bundle whose sections are the smooth vector fields over $M$ which are of uniformly of finite length with respect to $g$. Near the boundary, due to the form of the metric \eqref{1.3}, they are spanned over $C^\infty(M)$ by the vector fields $x^2 \partial_x$ and $x \partial_{y_i}$. Dually, the scattering cotangent bundle is spanned near the boundary by vector fields $dx/x^2 = -d(1/x)$ and $dy_i/x$; away from the boundary, it is canonically diffeomorphic to the usual cotangent bundle. Thus, a point in the scattering cotangent bundle can be expressed as a linear combination of
\begin{equation}
\nu \lambda d\big( \frac1{x} \big) + \sum_{i=1}^{n-1} \lambda \mu_i \frac{dy_i}{x}
\label{numu}\end{equation}
near the boundary, or
\begin{equation}
 \sum_{i=1}^{n} \lambda \zeta_i dz_i
\label{zeta}\end{equation}
away from the boundary, which defines linear coordinates $(\mu, \nu)$ or $\zeta$ on each fibre of the scattering cotangent bundle.
Notice that we have introduced a scaling by the spectral parameter $\lambda$; as $\lambda = 1/h$ this is essentially the semiclassical scaling, appropriate to our operator $\Delta - \lambda^2 = \lambda^2 (h^2 \Delta - 1)$, although in this low energy case, we are looking at the limit $h \to \infty$, rather than $h \to 0$ as in the high energy case in Section~\ref{sec:high}.

The appropriate `compressed cotangent bundle' over $M^2_{k,b}$ is discussed in \cite[Section 2.3]{GHS1}. Here, we only describe this for $\lambda > 0$ plus a neighbourhood of the boundary hypersurface $\bfc$. In this region, it is given by the lift of the bundle ${}^{\sc} T^* M \times {}^{\sc} T^* M$ to $M^2 \times [0, 1]$ and then to $M^2_{k,b}$. In particular, we use coordinates $(\mu, \nu)$ lifted from the left factor of $M$ and $(\mu', \nu')$ lifted from the right factor of $M$ in a neighbourhood of $\bfc$. We remark that these coordinates remain valid in a neighbourhood of $\bfc$ even at $\lambda = 0$, which follows from the fact that \eqref{numu} can be written in the form
$$
\nu d\big( \frac1{\rho} \big) + \sum_{i=1}^{n-1} \mu_i \frac{dy_i}{\rho}.
$$

 \

The following lemma will be useful in our estimates in Section~\ref{sec:localized}.

\begin{lemma}\label{lem:dist} Let $w = (w_1, \dots, w_n)$ denote a set of defining functions for $\diagb \subset \MMkb$; that is, the differentials $dw_i$ are linearly independent, and $\diagb = \{ w = 0 \}$. For example, near $\bfo$ or $\bfc$, we can take $w = (\sigma - 1, y_1 - y'_1, \dots,  y_{n-1} - y'_{n-1})$. Then $|w|/x$ is comparable to $d(z,z')$ in a neighbourhood of $\diagb$. Equivalently, $|w|/\rho$ is comparable to $\lambda d(z,z')$.
\end{lemma}

\begin{proof} Away from $\bfo \cup \bfc$, $|w|^2$ is a quadratic defining function for $\diagb$, and so is $d(z,z')^2$, hence they are comparable. Now consider what happens near $\bfo$ or $\bfc$. Using coordinates $w = (\sigma - 1, y_1 - y'_1, \dots y_{n-1} - y'_{n-1})$, we have
$$
\frac{|w|}{x} \sim \Big| \frac{\sigma - 1}{x} \Big| + \Big| \frac{y-y'}{x} \Big|.
$$
Write $r = 1/x$; then this is
$$
| r - r' | + r |y-y'|.
$$
Given that the metric takes the form $dr^2 + r^2 h(x, y, dy)$, where $h$ is positive definite, we see that this is comparable to $d(z,z')$.
\end{proof}

\begin{remark} In the case $M^\circ = \RR^n$, with Euclidean coordinates $z = (z_1, \dots, z_n)$, we can take $w = (z_1 - z_1', \dots, z_n - z_n')$.
\end{remark}

\subsection{Pseudodifferential operators on the low energy space}
We use the class of pseudodifferential operators $\Psi^m_k(M; \Omegakbh)$ on $\MMksc$ introduced by Guillarmou and the first author. By definition, these operators have Schwartz kernels which are half-densities conormal to the diagonal $\diagsc$, smooth on $\MMksc$ away from the diagonal, and rapidly decreasing at all boundary hypersurfaces not meeting the diagonal, i.e. at $\lbo$, $\rbo$, $\lb$ and $\rb$. In addition, we will only consider those operators with kernels supported where $\rho, \rho' \leq C < \infty$. In this setting we can write the kernel in the form
\begin{equation}
\lambda^n \int e^{i\lambda/x\big( (1-\sigma) \nu + (y-y')\cdot \mu \big)} a(\lambda, \rho, y, \mu, \nu) \, d\mu \, d\nu \Big| \frac{dg dg' d\lambda}{\lambda} \Big|^{1/2}
\label{psi-le}\end{equation}
where $a$ is a classical symbol of order $m$ in the $(\mu, \nu)$ variables, smooth in $(\lambda, \rho, y)$ and supported where $\rho \leq c$.  We remark that the $|d\lambda/\lambda|^{1/2}$ factor is purely formal; if we write this in the form $A(z, z', \lambda) |dg dg' d\lambda/\lambda|^{1/2}$, then the action on a half-density $f |dg|^{1/2}$ is given by
$$
\Big( \int A(z, z', \lambda) f(z') dg(z')  \Big) |dg(z)|^{1/2}.
$$

From this representation it is easy to see the following

\begin{lemma}\label{ledb}
If $A \in \Psi^{m}_k(M; \Omegakbh)$ then $(\lambda \partial_\lambda)^N A$ is also a pseudodifferential operator of order $m$, i.e. $(\lambda \partial_\lambda)^N A \in \Psi^{m}_k(M; \Omegakbh)$.
\end{lemma}

\begin{proof}
It suffices to prove for $N=1$ and use induction. If $\lambda \partial_\lambda$ hits the function $a$ in \eqref{psi-le}, then $a$ is still a symbol of order $m$ in the $(\mu, \nu)$ variables, smooth in $(\lambda, \rho, y)$ and supported where $\rho \leq c$. (Notice that $\rho = x/\lambda$ depends on $\lambda$ as well.) On the other hand, if $\lambda \partial_\lambda$ hits the phase, this is the same as $\nu \partial_\nu + \mu \cdot \partial_\mu$ hitting the phase, as it is homogeneous of degree $1$ in both $\lambda$ and in $(\nu, \mu)$. Integrating by parts we obtain another symbol $\tilde a$ of order $m$. This completes the proof.
\end{proof}

\begin{lemma}\label{lekb}
If $A \in \Psi^{m}_k(M; \Omegakbh)$, and if $m < -n$, then $A$ satisfies a kernel bound
$$
\Big| A(z, z') \Big| \leq \lambda^n \big( 1 + \lambda d(z,z') \big)^{-N}
$$
for any $N \in \NN$.
\end{lemma}

\begin{proof}
If the order $m$ is less than $-n$, then the integral \eqref{psi-le} is absolutely convergent,
showing  that the kernel of $\lambda^{-n} A$ is uniformly bounded. Next, we note that the differential operator
$$
\frac{1 - \partial_\nu^2 - \sum_i \partial_{\mu_i}^2}{ 1 + \lambda^2
\big( x^{-2} (\sigma - 1)^2  + x^{-2} |y-y'|^2 \big)}
$$
leaves the exponential in \eqref{psi-le} invariant. By applying this $N$ times to the exponential and then integrating by parts, we see that the integral is bounded by
$$
C_N  \Big( 1 + \lambda^2  \big( x^{-2} (\sigma - 1)^2  + x^{-2}
|y-y'|^2\big) \Big)^{-N}
$$
for any $N$. Finally, as in the proof of Lemma~\ref{lem:dist},  the square of the Riemannian distance on $M$ is comparable to
$$
\frac{(\sigma - 1)^2}{x^2} + \frac{|y-y'|^2}{x^2},
$$
so the integral is bounded by $C_N (1 + \lambda d(z,z'))^{-N}$ for
any $N$.
\end{proof}

\begin{corollary}\label{cor:bdd-le}
If $A \in \Psi^{m}_k(M; \Omegakbh)$, and if $m < -n$, then $A$ is bounded $L^2(M^\circ) \to L^2(M^\circ)$ uniformly as $\lambda \to 0$. The same is true for $(\lambda \partial_\lambda)^N A$ for any $N$.
\end{corollary}

\begin{proof} This follows from the kernel bound in Lemma~\ref{lekb}, the volume estimate $c r^n \leq V(z, r) \leq C r^n$
for the volume $V(z, r)$ of the ball of radius $r$ centered at $z
\in M^\circ$, and Schur's test.
\end{proof}

\subsection{Low energy partition of the identity}\label{subsec:lepoi} Recall that, in Proposition~\ref{prop:localized spectral measure}, we employ a partition of the identity.
We use essentially the same partition of the identity as in
\cite{GHS2}. To define it, we specify the symbols of these
operators, which must form a partition of unity on the phase space.
We point out that, in our approach, it is crucial to be able to
localize in phase space (and hence necessary to use
pseudodifferential operators) in order to eliminate difficulties
with conjugate points.

For low energies, this partition is defined as follows. We choose a
function $\chi \in C^\infty(\RR)$ of a real variable, with $\chi(t)
= 0$ for $t \leq \epsilon$ and $\chi(t) = 1$ for $t \geq 2\epsilon$.
We define $Q^{\mathrm{low}}_0(\lambda)$ to be multiplication by the
function $1 - \chi(\rho)$ (recall $\rho = x/\lambda$). Next, we
choose $Q_1'(\lambda)$ such that its (full) symbol  is equal to $0$
for $1/2\leq |\mu|_h^2+\nu^2\leq3/2$, and equal to $1$ outside
$1/4\leq |\mu|_h^2+\nu^2\leq2$. Then we define
$Q^{\mathrm{low}}_1=\chi(\rho)Q_1'$. This means that the symbol of
$\Id - Q^{\mathrm{low}}_0 - Q^{\mathrm{low}}_1$ is supported where
$\rho$ is small and close to the characteristic variety
$|\mu|_h^2+\nu^2 = 1$. We then decompose this as $Q^{\mathrm{low}}_2
+ \dots +Q^{\mathrm{low}}_{N_l}$ such that the symbol of each
$Q^{\mathrm{low}}_j$, $j \geq 2$ has support where $\nu$ is
contained in a small interval.

\subsection{Localized spectral measure}\label{subsec:loc-spec-meas-le}
The main result of \cite{GHS1} was that the spectral measure for the
Laplacian on an asymptotically conic manifold is, for low energies,
a Legendre distribution associated to a pair of Legendre
submanifolds, the `propagating Legendrian' $L^{\bfc}$ and the
`incoming/outgoing Legendrian' $L^\sharp$.
We now explain very briefly what this means. We first have to introduce the contact manifold in which these Legendre submanifolds live.  Consider the bundle
${}^\Phi T^* M^2_b$, obtained by lifting ${}^{\sca} T^* M \times {}^{\sca} T^* M$
(viewed as a bundle over $M^2$) to $M^2_b$. This bundle carries a symplectic structure, but the symplectic form degenerates at the boundary. Nevertheless, it determines a contact structure on this bundle restricted to the boundary hypersurface $\bfc$\footnote{We denote the new boundary hypersurface of $M^2_b$, created by the blowup \eqref{M2b}, by $\bfc$. This is slightly at variance with the way $\bfc$ is used as a boundary hypersurface of $M^2_{k,b}$ --- here it really corresponds to taking a section of $M^2_{k,b}$ at fixed $\lambda_* > 0$ --- but hopefully no confusion will be caused.}, which we denote ${}^\Phi T^*_{\bfc} M^2_b$.  We give this contact structure in local coordinates $(y, y', \sigma, \mu, \mu', \nu, \nu')$ for ${}^\Phi T^*_{\bfc} M^2_b$, where $\sigma = x/x'$, $(\mu, \nu)$ are as in \eqref{numu}, and as above, the unprimed/primed coordinates are lifted from the left/right copies of ${}^{\sca} T^* M$. In these coordinates,  the contact form has an expression
$$
d\nu - \mu \cdot dy + \sigma (d\nu' - \mu' \cdot dy').
$$
A Legendrian submanifold is, by definition, an $2n-1$-dimensional submanifold of this $4n-1$-dimensional space on which the contact form vanishes. The Legendre submanifold $L^\sharp$ is easy to define: it is the submanifold
$$
\{ (y, y', \sigma, \mu, \mu', \nu, \nu') \mid \mu = \mu' = 0, \ \nu = \nu' = 1 \}.
$$

The other Legendre submanifold, $L^{\bfc}$, is more interesting.
It encodes the geodesic flow on the cone over $(\partial M, h)$ where $h = h(0)$ is the metric in \eqref{1.3}. Let  $(y, \eta)$ be an element of the cosphere bundle $S^* \partial M$ of $T^* \partial M$ and $\gamma(s) = (y(s), \eta(s))$ be the geodesic with $(y(0), \eta(0)) = (y, \eta)$. Then  $L^{\bfc}$ is given by the union of the leaves $\gamma^2 = \gamma^2(y, \eta)$,
\begin{multline}
\gamma^2 = \textrm{clos } \big\{ (y, y', \sigma = x/x', \mu, \mu', \nu, \nu') \mid y = y(s), y' = y(s'), \mu = \eta(s) \sin s, \\ \mu' = -\eta(s') \sin s', \nu = -\cos s, \nu' = \cos s', \sigma = \sin s/\sin s', (s, s') \in (0, \pi)^2 \big\}
\label{gamma^2}\end{multline}
as $(y, \eta)$ ranges over $S^* \partial M$. We note that this closure includes the sets
\begin{equation}
T_\pm = \big\{ (y, y', \sigma, \mu, \mu', \nu, \nu') \mid y=y', \ \sigma \in \RR, \  \mu = \mu' = 0, \ \nu = -\nu' = \pm 1 \},
\label{Tpm}\end{equation}
corresponding to the limit $s, s' \to 0$ and $s, s' \to \pi$.

The statement that the spectral measure is a Legendre distribution with respect to the pair of Legendre submanifolds $(L^{\bfc}, L^\sharp)$ means that the Schwartz kernel of the spectral measure can be expressed as an oscillatory function or oscillatory integral, with a phase function that `parametrizes' the Legendre submanifold. We now state what `parametrizes' means, first in the case of a Legendre submanifold $L$ that projects diffeomorphically to the base $\bfc$, in the sense that the projection from ${}^\Phi T^* _{\bfc}M^2_b$ to $\bfc$ restricts to a (local) diffeomorphism from $L$ to $\bfc$. In this case, there exists a function $\Phi : \bfc \to \RR$ such that (locally) $L$ is the graph of the differential of the function $\Phi/x$, or in coordinates,
\begin{multline*}
L = \{ \mu = d_y \Phi(y, y', \sigma), \ \mu' = \sigma^{-1} d_{y'} \Phi(y, y', \sigma), \\ \nu = \Phi(y, y', \sigma) - \sigma d_\sigma \Phi(y, y', \sigma), \ \nu' = d_\sigma \Phi(y, y', \sigma) \}.
\end{multline*}
We say that $\Phi$, or more accurately $\Phi/x$,  (locally) \emph{parametrizes} $L$.
In the general case, there always exist (nonunique) functions $\Phi(y, y', \sigma, v)$, depending on extra variables $(v_1, \dots, v_k)$,
that locally parametrize $L$ in the sense that
\begin{multline}
L = \{ \mu = d_y \Phi(y, y', \sigma,v), \ \mu' = \sigma^{-1} d_{y'} \Phi(y, y', \sigma,v), \\ \nu = \Phi(y, y', \sigma,v) - \sigma d_\sigma \Phi(y, y', \sigma,v), \ \nu' =  d_\sigma \Phi(y, y', \sigma,v) \mid d_v \Phi = 0 \}.
\label{LPhiparam}\end{multline}

Observe that if we take the union of the points of \eqref{gamma^2} with $s=s'$, over all $(y, \eta)\in S^* \partial M$, then we get a codimension one submanifold of $L^{\bfc}$, which is also a codimension one submanifold of the conormal bundle of the diagonal $N^* \diagb$, given by
$$
N^* \diagb = \{ (y, y', \sigma, \mu, \mu', \nu, \nu') \mid y=y', \ \sigma = 1, \ \mu = -\mu', \ \nu = -\nu' \}.
$$
\emph{It turns out that
in a deleted neighbourhood of $N^* \diagb$, $L^{\bfc}$ projects
in a 2:1 fashion to the base $\bfc$, i.e. $L^{\bfc} \setminus N^* \diagb$ consists of 2 sheets, each of which project diffeomorphically to the base $\bfc$, and which are parametrized by the function $\pm d_{\conic}$, where $d_{\conic}$ is  the distance function  on the cone over $\partial M$.}
 The conic metric $d_{\conic}$ has an explicit expression when
$d_{\partial M}(y,y') < \pi$.  Writing $r = 1/x$, $r' = 1/x' =
\sigma /x$, it takes the form
\begin{equation}
d_{\conic}(y,y', r, r') = \sqrt{ r^2 + {r'}^2 - 2r r' \cos
d_{\partial M}(y, y')} = r \sqrt{1 + \sigma^2 -2\sigma \cos
d_{\partial M}(y, y')}. \label{dconic}\end{equation}
Notice that $d_{\conic}(y,y', r, r')/r$ indeed has the form $\Phi(y, y', \sigma)/x$, and is smooth provided that $\cos d_{\partial M}(y, y')$ is smooth, i.e. $d_{\partial M}(y, y')$ is less than the injectivity radius on $(\partial M, h)$.

We next explain why we consider the localized (or more precisely microlocalized) spectral measure, by which we mean any of the operators $Q(\lambda) \specl Q(\lambda)^*$ where $Q(\lambda)$ is a member of our partition of the identity. The reason is, as shown in \cite[Section 5]{GHS2}, these terms are also Legendre distributions, but associated only to part of the Legendrian, namely to the subset
$$
\{ (y, y', \sigma, \mu, \mu', \nu, \nu') \in L \mid (y, \mu, \nu) , (y', \mu', \nu') \in WF'(Q),
\}
$$
where $WF'(Q)$ is the support of the symbol\footnote{the relevant symbol here is the scattering symbol, or boundary symbol, in the scattering calculus, which is a function on ${}^\Phi T^*_{\bfc} M^2_b$; see \cite{Mel}.} of $Q$. This is localized close to $N^* \diagb \cup T_\pm$ (that is, those points in \eqref{gamma^2} corresponding to $s=s'$), if $WF'(Q)$ is well localized. We can then use the italicized statement above to write this piece of the spectral measure using the conic distance function, except near $N^* \diagb$ itself, where we can express it as an oscillatory integral using a slightly more complicated form of phase function (as in part (ii) of Proposition~\ref{prop:osc-form-low}).

We summarize the information we need from \cite{GHS1}, \cite{GHS2} concerning the spectral measure:


\begin{proposition}\label{prop:osc-form-low}
Let $Q^{\mathrm{low}}_j(\lambda)$ be a member of the partition of the identity
defined above. Let $\eta > 0$ be given. Then for $j,k=0$ or $1$,  $Q^{\mathrm{low}}_j(\lambda) \specl Q^{\mathrm{low}}_k(\lambda)^*$ satisfies the estimates on the RHS of
\eqref{beans}; and $Q^{\mathrm{low}}_j(\lambda) \specl
Q^{\mathrm{low}}_j(\lambda)^*, j\geq2$ can be written as a finite sum of terms of
the following two types:

(i) An oscillatory function of the form
\begin{equation}
 \lambda^{n-1} e^{\pm i\lambda d_{\conic}(y, y', \frac1{x}, \frac{\sigma}{ x})} a(y, y', \sigma, x, \lambda)
\label{a}\end{equation}
where $a$ is supported where $x, x' \leq \eta$ and $d_{\partial M}(y, y') \leq \eta$ and satisfies estimate \eqref{bean};

(ii) An oscillatory integral of the form
\begin{equation}
\lambda^{n-1} \int_{\RR^{n-1}} e^{i\Phi(y, y', \sigma, v)/\rho} \tilde a(y, y', \sigma, v, \rho, \lambda) \, dv
\label{lsmoi}\end{equation}
where $\tilde a$ is smooth in all its arguments, and supported in a small neighbourhood of a point $(y_0, y_0, 1, v_0, 0, 0)$ such that $d_v\Phi(y_0, y_0, 1, v_0) = 0$.
Moreover, writing  $w = (w_1, \dots, w_n)$ for a set of coordinates  defining $\diagb \subset \MMkb$, i.e. $w = (y-y', \sigma - 1)$, and $v = (v_2, \dots, v_n)$, one can rotate in the $w$ variables such that  the function $\Phi = \Phi(y, w, v)$ has the properties
\begin{equation}\begin{gathered}
\begin{cases}
\mathrm{(a) \ } d_{v_j} \Phi = w_j + O(w_1), \\
\mathrm{(b) \ }  \Phi = \sum_{j=2}^n v_j d_{v_j} \Phi + O(w_1), \\
\mathrm{(c) \ }  d^2_{v_jv_k} \Phi = w_1 A_{jk}, \\
\mathrm{(d) \ }  d_v \Phi = 0 \implies \Phi/x = \pm   d_{\conic}(y,y', \frac1{x}, \frac{\sigma}{ x})
\end{cases}
\end{gathered}\label{Phi-properties}\end{equation}
where $A_{jk}$ is nondegenerate for all $(y, w, v)$ in the support of $b$. Here $d_{\conic}$ is as in \eqref{dconic}.  \end{proposition}

\begin{proof}
The statement about $Q^{\mathrm{low}}_j(\lambda) \specl Q^{\mathrm{low}}_k(\lambda)^*$ for $j,k=0,1$, follows from the microlocal support estimates in \cite[Section 5]{GHS2}. In fact, $Q^{\mathrm{low}}_0(\lambda)$ has empty wavefront set, while $Q^{\mathrm{low}}_1(\lambda)$ has wavefront set disjoint from the characteristic variety of $\mathbf{H} - \lambda^2$, which contains the microlocal support of $\specl$. It follows that the operators $Q^{\mathrm{low}}_j(\lambda) \specl Q^{\mathrm{low}}_k(\lambda)^*$, for $j,k=0,1$, vanish rapidly at $\bfc$, $\lb$ and $\rb$. Also, as shown in \cite{GHS1}, $\specl$ is polyhomogeneous at the other boundary hypersurfaces of $\MMkb$, namely $\zf, \lbo, \rbo$ and $\bfo$, vanishing to order $n-1$ at each of these faces. Since the $Q^{\mathrm{low}}_j(\lambda)$ are pseudodifferential operators of order zero, the same is true of the composition $Q^{\mathrm{low}}_j(\lambda) \specl Q^{\mathrm{low}}_k(\lambda)^*$, for $j,k=0,1$ (see \cite[Lemma 5.2]{GHS2}). To translate this into an estimate, we observe that $\lambda$ is a product of boundary defining functions for $\zf, \lbo, \rbo$ and $\bfo$, while  a product of boundary defining functions for $\bfc$, $\lb$ and $\rb$ is 
$O((1 + \lambda d(z,z'))^{-1})$. 
The estimate \eqref{beans} follows directly.

We next discuss (i) and (ii). Everything in this statement has been proved in \cite[Lemma 6.5 and Proposition 6.2]{GHS2} except for the statement that $\Phi$ is given by the conic distance function when $d_v \Phi = 0$. To see this, we use
the explicit formula \eqref{dconic} for the conic distance function, the relation \eqref{LPhiparam}, and the  description of the  Legendre submanifold $L^{\bfc}$ in \eqref{gamma^2}. From \eqref{LPhiparam}, it follows that $\Phi = \nu + \sigma \nu'$. Writing $\nu$ and $\nu'$ in terms  of $s$ and $s'$, using \eqref{gamma^2}, we see that
$$
d_v \Phi = 0 \implies \Phi = - \cos s + \sigma \cos s'.
$$
If we square this then we get
$$
d_v \Phi = 0 \implies \Phi^2 = \cos^2 s + \sigma^2 \cos^2 s' - 2\sigma \cos s \cos s'.
$$
We can write the RHS in the form
$$
1 - \sin^2 s + \sigma^2 (1 - \sin^2 s') - 2 \sigma \big( \cos(s-s') - \sin s \sin s' \big).
$$
Noting that $\sin^2 s + \sigma^2 \sin^2 s' = 2 \sigma \sin s \sin s'$, using the expression for $\sigma$ in \eqref{gamma^2}, we see  that
$$
d_v \Phi = 0 \implies \Phi^2 = 1 + \sigma^2 - 2 \sigma \cos d_{\partial M}(y,y') .
$$

\end{proof}

\begin{remark} It might help to give an example to show how \eqref{Phi-properties} works. In Euclidean space,
the Schwartz kernel of the spectral measure
$dE_{\sqrt{\Delta}}(\lambda)$ of $\sqrt{\Delta}$ is given by
$$dE_{\sqrt{\Delta}}(\lambda;z,z')=\frac{\lambda^{n-1}}{(2\pi)^n}\int_{\mathbb{S}^{n-1}} e^{i\lambda(z-z')\cdot\zeta}d\zeta,$$
one can find the phase function $(z - z') \cdot \zeta$, where $\zeta
\in \mathbb{S}^{n-1}$. Locally near $\zeta = (1, 0, \dots, 0)$, we
can write $\zeta = (\sqrt{1 - |v|^2}, v_2, \dots, v_n)$. Write $x =
|z|^{-1}$ and  $w = (z-z')/|z|$. Then the phase function becomes
$$
\Phi =  w_1\sqrt{ 1 - v_2^2 - \dots - v_n^2} + \sum_{j=2}^n w_j v_j ,
$$
and we can check that properties $\mathrm{(a)}-\mathrm{(d)}$ of
\eqref{Phi-properties} hold in this case.
\end{remark}


\section{Spectral measure and partition of the identity at high energies}\label{sec:high}
In the previous section we recalled the partition of the identity
operator and the structure of the localized spectral measure for low energy i.e.
$0<\lambda\le \lambda_0$. We now do the same for high energies,
$\lambda\in[\lambda_0,\infty)$. For the sake of convenience, we
introduce the semiclassical parameter $h=\lambda^{-1}$ (which should
not be confused with $h$ in the metric $g$), so that we pay our
attention to the range $h\in(0,h_0]$, where
$h_0=\lambda_0^{-1}$. The spectral measure of the operator
$\mathbf{H}$ for high energy was constructed in \cite{HW} on the
high energy space $\mathbf{X}$. Our main task is to adapt each of
main results in the previous section to the high energy setting.

\subsection{High energy space} The high energy $\mathbf{X}$, introduced in   \cite{HW}, is defined by $\mathbf{X}=[0,h_0]\times M_b^2$, where
$M_b^2=[M^2;\partial M\times\partial M]$ is as in \eqref{M2b}. We label the boundary hypersurfaces in
$\mathbf{X}$ by $\rb$, $\lb$, $\bfc$ and $\mf$, according as they are the
lifts to $\mathbf{X}$ of the faces
$$[0,h_0]\times M\times \partial M, \quad [0,h_0]\times \partial M\times
M, \quad [0,h_0]\times \partial M\times \partial M, \text{ or } \{0\}\times M^2$$ of $[0, h_0] \times M^2$, respectively. The labelling of
boundary hypersurfaces is consistent with the notations defined in
the low energy space, since when $\lambda\in(C^{-1},C)$ (where
$\lambda=1/h$) the spaces both have the form $(C^{-1},C)\times
M_b^2$. Recall $\sigma=x/x'$, we can use coordinates
$(y,y',\sigma,x',h)$ near $\mathrm{bf}$ and away from $\mathrm{rb}$,
and coordinates $(y,y',\sigma^{-1},x,h)$ near $\mathrm{bf}$ and away
from $\mathrm{lb}$. We use coordinates $(z,z',h)$ away from
$\mathrm{bf}$, $\mathrm{rb}$ and $\mathrm{lb}$.

\subsection{Semiclassical scattering pseudodifferential operators} We recall
the space
$\Psi_{\mathrm{sc},h}^{m,l,k}(M;\leftidx{^{s\Phi}}\Omega^{1/2})$ of
semiclassical scattering pseudodifferential operators, introduced by
Wunsch and Zworski \cite{WZ} based on Melrose's scattering calculus \cite{Mel}.
Such operators are indexed by the differential order $m$, the boundary
order $l$ and the semiclassical order $k$ .  One can express this
space in terms of the space with $l=k=0$ by
$$\Psi_{\mathrm{sc},h}^{m,l,k}(M;\leftidx{^{s\Phi}}\Omega^{1/2})=x^lh^{-k}\Psi_{\mathrm{sc},h}^{m,0,0}(M;\leftidx{^{s\Phi}}\Omega^{1/2}).$$
The Schwartz kernel of semiclassical pseudodifferential operator
$A\in
\Psi_{\mathrm{sc},h}^{m,0,0}(M;\leftidx{^{s\Phi}}\Omega^{1/2})$
takes the following form on $\mathbf{X}$. Near the diagonal $\diagb
\subset M^2_b$ and away from $\bfc$, it takes the form
\begin{equation}
h^{-n}\int e^{i(z-z')\cdot\zeta/h}a(z,\zeta,h)d\zeta \ \Big| \frac{dg dg' dh}{h^2} \Big|^{1/2}, \quad n = \dim M,
\label{psi-he-int}\end{equation}
while near the boundary of the diagonal, $\diagb \cap \bfc$, it takes the form
\begin{equation}\label{psi-he}
h^{-n}\int
e^{i((y-y')\cdot\mu+(\sigma-1)\nu)/(hx)}a(x,y,\mu,\nu,h)d\mu d\nu \
\Big| \frac{dg dg' dh}{h^2} \Big|^{1/2}
\end{equation}
Here $a$ is a symbol of order $m$ in the variable $\zeta$
or $(\eta,\nu)$ variables and is smooth in the remaining variables.
Finally, away from $\diagb$, the kernel of $A$ is smooth and vanishes to all orders at $\bfc$, $\lb$, $\rb$ and $\mf$.

\begin{lemma}\label{hedb}
If $A\in
\Psi_{\mathrm{sc},h}^{m,0,0}(M;\leftidx{^{s\Phi}}\Omega^{1/2})$ then
$(h
\partial_h)^N A$ is also a pseudodifferential operator of
order $m$, i.e. $(h \partial_h)^N A \in
\Psi_{\mathrm{sc},h}^{m,0,0}(M;\leftidx{^{s\Phi}}\Omega^{1/2})$.
\end{lemma}

\begin{proof} Away from the diagonal, the result is trivial, as the kernel is smooth and $O(h^\infty)$. So consider the representations \eqref{psi-he-int} and \eqref{psi-he}.
The proof is parallel to the argument in Lemma
\ref{ledb}.
By induction, we only need consider $N=1$. If $h
\partial_h$ hits the function $a$ in \eqref{psi-he}, then $a$
is still a symbol of order $m$ in the $(\eta, \nu)$ variables,
smooth in $(h, x, y)$ and supported where $xh \leq c$. On the other
hand, if $h \partial_h$ hits the phase, this is the same as $\nu
\partial_\nu + \eta \cdot \partial_\eta$ hitting the phase, as it
brings a factor which is homogeneous of degree $-1$ in $h$ and of
degree $1$ in $(\nu, \eta)$. Integrating by parts we obtain another
symbol $\tilde a$ of order $m$. The argument for \eqref{psi-he-int} is analogous.
This completes the proof.
\end{proof}

\begin{lemma}\label{hekb}
If $A
\in\Psi_{\mathrm{sc},h}^{m,0,0}(M;\leftidx{^{s\Phi}}\Omega^{1/2})$,
and if $m < -n$, then $A$ satisfies a kernel bound
\begin{equation}
\Big| A(z, z') \Big| \leq h^{-n} \big( 1 + h^{-1}d(z,z') \big)^{-N}
\label{Azz'}\end{equation}
for any $N \in \NN$.
\end{lemma}

\begin{proof} This estimate is straightforward  away from the diagonal, as the Schwartz kernel of $A$ vanishes rapidly at all boundaries away from the diagonal. On the other hand, the RHS is a positive multiple of $h^{N-n} \rho_{\lb}^N \rho_{\bfc}^N \rho_{\rb}^N$ away from the diagonal.

Near the diagonal, we have the representations \eqref{psi-he-int} and \eqref{psi-he}.
The argument in the same sprit as Lemma \ref{lekb}.
If the order $m$ is less than $-n$, then the integral \eqref{psi-he}
is absolutely convergent, showing  that the kernel of $h^{n} A$ is
uniformly bounded. Next, we note that the differential operator
$$
\frac{1 - \partial_\nu^2 - \sum_i \partial_{\eta_i}^2}{ 1 + \big(
(hx)^{-2} (\sigma - 1)^2  + (hx)^{-2} |y-y'|^2 \big)}
$$
leaves the exponential in \eqref{psi-he} invariant. By integrating
by parts N-times, we see that the integral is bounded by
$$
C_N  \Big( 1 +  \big( (hx)^{-2} (\sigma - 1)^2  + (hx)^{-2}
|y-y'|^2\big) \Big)^{-N}
$$
for any $N$. Finally, we note that the square of the Riemannian
distance on $M$ is comparable to
$$
\frac{(\sigma - 1)^2}{x^2} + \frac{|y-y'|^2}{x^2},
$$
so the integral is bounded by $C_N (1 + h^{-1}d(z,z'))^{-N}$ for any
$N$.
\end{proof}

\begin{corollary}\label{cor:bdd-he}
If $A
\in\Psi_{\mathrm{sc},h}^{m,0,0}(M;\leftidx{^{s\Phi}}\Omega^{1/2})$,
and if $m < -n$, then $A$ is bounded $L^2(M^\circ) \to L^2(M^\circ)$
uniformly as $h \to 0$. The same is true for $(h
\partial_h)^N A$ for any $N$.
\end{corollary}

\begin{proof} This follows from the kernel bound in Lemma~\ref{hekb} and Schur's test,
since there is a uniform volume estimate $c r^n \leq V(z, r) \leq C r^n$ for the
volume $V(z, r)$ of the ball of radius $r$ centred at $z \in
M^\circ$.
\end{proof}

\subsection{High energy partition of the identity}\label{subsec:partition-he} We now describe the partition of the identity used in Proposition~\ref{prop:localized spectral measure} for high energies. Similar to before, these operators are obtained by quantizing symbols which form a partition of unity (independent of $h$) in the scattering cotangent bundle, ${}^{\sca}T^* M$.
We first choose the symbol of
$Q^{\mathrm{high}}_1$ to vanish where $|\mu|_h^2 + \nu^2 \in [1/2, 3/2]$, and to be identically one where $|\mu|_h^2 + \nu^2 \leq 1/4$ or $|\mu|_h^2 + \nu^2 \geq 2$.  Noting that the symbol of
$\mathrm{Id}-Q^{\mathrm{high}}_1$ is supported close  to the
characteristic variety of $h^2\Delta_g - 1$, that is, the set $\{ |\mu|_h^2 + \nu^2 = 1 \}$, we decompose
$\mathrm{Id}-Q^{\mathrm{high}}_1$ as
$Q^{\mathrm{high}}_2+\cdots+Q^{\mathrm{high}}_{N'}+Q^{\mathrm{high}}_{N'+1}+\cdots+Q^{\mathrm{high}}_{N}$
such that the symbol of  $Q^{\mathrm{high}}_j$, for $j = 2 \dots N'$,  is supported in the set
$$
\{ x \leq \epsilon,  \ \nu \in B_j \}
$$
where the sets $B_j \subset [-2, 2]$  are sufficiently small open intervals with union $[-2, 2]$. For $j \geq N' + 1$, we choose the symbols of $Q^{\mathrm{high}}_j$ so that each one has small support in the interior of $T^* M^\circ \cap
\{ x \geq \epsilon/2 \}$.
Note that we may (and will) assume that $N' = N_l$ and that
$Q^{\mathrm{high}}_j(\lambda) = Q^{\mathrm{low}}_j(\lambda)$ for
intermediate energies, and for $1 \leq j \leq N_l$.

\subsection{Localized spectral measure} In \cite{HW}, Wunsch and the first author showed that the spectral measure
for the Laplacian on this setting is, for high energy, a
Legendre distribution associated to a pair of Legendre submanifolds, $L$ and $L^\sharp$. We briefly explain the meaning of this statement.
The Legendre submanifold $L^\sharp$ has already been defined in Section~\ref{subsec:loc-spec-meas-le}; it lives in the contact manifold ${}^\Phi T^*_{\bfc} M^2_b$, living over the boundary hypersurface $\bfc$. The new Legendre submanifold $L$ encodes the geodesic flow on $T^* M^\circ$. It is a submanifold of $\RR \times {}^{\Phi} T^* M^2_b$, which has a natural contact form, described as follows. We write $\alpha$ for the contact form on ${}^{\sca} T^* M$ induced by the inclusion of $T^* M^\circ$ into ${}^{\sca} T^* M$, and $\alpha, \alpha'$ for the lift of this contact form to ${}^{\Phi} T^* M^2_b$ by the left, resp. right projections. Writing $\tau$ for the coordinate on the $\RR$-factor in $\RR \times {}^{\Phi} T^* M^2_b$, then the contact form  on this space takes the form
$$
\alpha + \alpha' - d\tau.
$$
Then $L$ is given as follows: let $\Sigma$ denote the characteristic variety of $h^2 \Delta_g - 1$, given in local coordinates by $\{ |\zeta|_{g(z)} = 1 \}$ in the interior or $\{ |\mu|_{h(x,y)}^2 + \nu^2 = 1 \}$ near the boundary. Then $L$ is given in terms of geodesic flow $G_t$ by
\begin{equation}
L =  \big\{ (q, q',  \tau) \mid q, q' \in \Sigma, \ q = G_\tau(q') \big\}
\label{L}\end{equation}
(this follows from \cite[Equation 7.9]{GHS2} and the discussion following). In $\RR \times {}^{\Phi} T^* M^2_b$, $L$ can be restricted to $\RR \times {}^{\Phi} T^*_{\bfc} M^2_b$, i.e. restricted to lie over $\bfc$, and then forgetting the $\tau$ component, we obtain the Legendre submanifold $L^{\bfc}$ from Section~\ref{subsec:loc-spec-meas-le}\footnote{The relation between the various Legendre submanifolds is explained in detail in \cite[Part 1]{HW}.}.

As in Section~\ref{subsec:loc-spec-meas-le}, the statement that an operator is Legendrian with respect to $L$ means that its Schwartz kernel can be expressed as an oscillatory function or oscillatory integral using a phase function that locally parametrizes $L$.
In the interior of $\mathbf{X}$, this means a function $\Psi(z, z', v)$ such that, locally,
using coordinates $(z, \zeta, z', \zeta', \tau)$ on $\RR \times {}^{\Phi} T^* M^2_b$, we have
$$
L = \big\{ (z, d_z\Psi, z', d_{z'} \Psi, \Psi)  \mid d_v \Psi = 0 \big\}.
$$
In particular,  $\tau$ is equal to the value of the phase function when $d_v \Psi = 0$.
If there are no $v$ variables, the condition $d_v \Psi = 0$ is omitted, and then $L$ is
(essentially) the graph of the differential of $\Psi$. Near the boundary $\bfc$, we use local coordinates $(x, y, y', \sigma, \mu, \nu, \mu', \nu', \tau)$ and then a local parametrization of $L$ is a function $\Psi(x, y, y', \sigma, v)/x$ such that
$$
L = \big\{ (x, y, y', \sigma, d_y \Psi, \Psi - x d_x \Psi, - \sigma d_\sigma \Psi, \sigma^{-1} d_{y'} \Psi, d_\sigma \Psi, \Psi) \mid d_v \Psi = 0 \big\}.
$$

We give some consequences of this result for the
localized spectral measure needed in this paper. As in the low
energy case, the localized spectral measure refers to any operator of the form
$Q^{\mathrm{high}}(\lambda)dE_{\sqrt{\mathbf{H}}}(\lambda)Q^{\mathrm{high}}(\lambda)^*$ where
$Q^{\mathrm{high}}(\lambda)$ is a member of the partition of the identity operator from Section~\ref{subsec:partition-he}. As above, we write $h = 1/\lambda$.

\begin{proposition}\label{prop:osc-form-high}
Let $Q^{\mathrm{high}}_j(\lambda)$ be a member of the partition of the identity
defined above. Then $Q^{\mathrm{high}}_j(\lambda) \specl Q^{\mathrm{high}}_j(\lambda)^*$ satisfies
\eqref{beans} for $j=1$, while for $j \geq 2$, it can
be written as a finite sum of terms of the following three types:

(i) An oscillatory function of the form
\begin{equation}
h^{-(n-1)} e^{\pm i d(z, z')/h} \tilde a(z, z', h)
\label{atilde}\end{equation} where $\tilde a$ satisfies estimate
\eqref{bean}.

(ii) An oscillatory integral supported where $x, x' \geq \epsilon$ of the form
\begin{equation}
h^{-(n-1)}\int_{\R^{n-1}}e^{
i\Psi(z,z',v)/h}b(z,z',v,h)dv,
\label{hsmoi2}\end{equation}
where $b$ is smooth in all its arguments, and supported in a small neighbourhood of a point $(z_0, z_0, v_0, 0)$ such that $d_v\Psi(z_0, z_0, v_0) = 0$. Moreover,
writing  $w = z-z'$, and $v = (v_2, \dots, v_n)$,
one can rotate in the $w$ variables such that  the function $\Psi = \Psi(z, w, v)$
has the properties
\begin{equation}\begin{gathered}
\begin{cases}
\mathrm{(a) \ } d_{v_j} \Psi = w_j + O(w_1), \\
\mathrm{(b) \ }  \Psi = \sum_{j=2}^n v_j d_{v_j} \Psi + O(w_1), \\
\mathrm{(c) \ }  d^2_{v_jv_k} \Psi = w_1 A_{jk}, \\
\mathrm{(d) \ }  d_v \Psi = 0 \implies \Psi(z,z', v) = \pm   d(z,z')
\end{cases}
\end{gathered}\label{Psi-properties}\end{equation}
where $A_{jk}$ is nondegenerate at $(z_0, z_0, v_0)$, and $d(z,z')$ is the Riemannian distance function on $M^\circ \times M^\circ$;

(iii) An oscillatory integral supported near $x = x' = 0$ of the form
\begin{equation}
h^{-(n-1)}\int_{\R^{n-1}}e^{
i\Psi(y,y',\sigma, x,v)/(hx)}b(y,y',\sigma,x,v,h)dv,
\label{hsmoi}\end{equation}
where $b$ is smooth in all its arguments, and supported in a small
neighbourhood of a point $(y_0, y_0, 1, 0,  v_0, 0)$ such that
$d_v\Psi(y_0, y_0, 1, v_0) = 0$. Moreover, writing  $w = (w_1,
\dots, w_n)$ for a set of coordinates  defining $\diagb\subset
M_b^2$, i.e. $w = (y-y', \sigma - 1)$, and $v = (v_2, \dots, v_n)$,
one can rotate in the $w$ variables such that  the function $\Psi = \Psi(y, w, x, v)$
has the properties
\begin{equation}\begin{gathered}
\begin{cases}
\mathrm{(a) \ } d_{v_j} \Psi = w_j + O(w_1), \\
\mathrm{(b) \ }  \Psi = \sum_{j=2}^n v_j d_{v_j} \Psi + O(w_1), \\
\mathrm{(c) \ }  d^2_{v_jv_k} \Psi = w_1 A_{jk}, \\
\mathrm{(d) \ }  d_v \Psi = 0 \implies \Psi/x = \pm   d(z,z')
\end{cases}
\end{gathered}\label{Psi-properties-2}\end{equation}
where $A_{jk}$ is nondegenerate at $(y_0, y_0, 1, 0,  v_0, 0)$.
\end{proposition}

\begin{remark} Indeed, noting that $\lambda=1/h$, this is an analogue of
Proposition \ref{prop:osc-form-low} for the case of
$\mathbf{X}=[0,h_0]\times M_b^2$.
\end{remark}

\begin{proof}
The proof is analogous to the proof of Proposition~\ref{prop:osc-form-low}, with the main difference being that the computation takes place over the whole of $M^2_b$ (including the interior), not just at the boundary as is the case in the low energy case.
We prove (ii), i.e.~ we work in the interior of $M^2_b$, using coordinates $(z, z')$, with $z$ a coordinate on the left copy of $M^\circ$, and $z'$ on the right copy. The proof for (iii) is only notationally different.

As in the low energy case, the Legendre submanifold $L$ has the property that it intersects $N^* \diagb$ in a codimension one submanifold, and in a deleted neighbourhood ofr $N^* \diagb$, it projects in a 2:1 fashion down to the base, $\mf = M^2_b$, such that the two sheets are parametrized by the phase functions $\pm d(z, z')$.

We now apply \cite[Lemma 7.6 and (ii) of Lemma 7.7]{GHS2}. This tells us that for any point in the microlocal support of $Q^{\mathrm{high}}_j(\lambda) \specl Q^{\mathrm{high}}_j(\lambda)^*$, either there is a neighbourhood in which $L$ projects diffeomorphically to the base $M^2_b$, or the point lies at the conormal bundle to the diagonal, i.e. $z=z'$, $\zeta = -\zeta'$. In the former case, the function $\pm d(z,z')$ can be used directly as the phase function, and we obtain the statement (i) in the Proposition. In the latter case, a phase function $\Psi$ depending on $n-1$ variables $v_2, \dots, v_n$ can be constructed following the general approach of \cite[Proposition 7.5]{GHS2}. Since this was not written down explicitly in the coordinates $(z,z')$ valid in the interior of $M^2_b$ we sketch briefly how this is done.
It follows from the proof of Lemma 7.6 of \cite{GHS2}, we can rotate coordinates so that $w_1, \zeta_2, \dots, \zeta_n, z'$ give coordinates on $L$ locally. (The proof of Lemma 7.6 shows that one can take $(\tau, \zeta_2, \dots, \zeta_n, z')$ but since it is also shown that $\partial z_1 / \partial \tau \neq 0$, then one can substitute $z_1$ for $\tau$, and then substitute $w_1 = z_1 - z_1'$ for $z_1$.) One can therefore express the  functions $w_2, \dots, w_n$, and $\tau$ on $L$ as smooth functions $W_j(w_1, \zeta_2, \dots, \zeta_n, z')$ and $T(w_1, \zeta_2, \dots, \zeta_n, z')$ of these coordinates. Then the function
$$
\Psi(w, z', v) = \sum_{j=2}^n (w_j - W_j(w_1, \zeta_2, \dots, \zeta_n, z')) v_j + T(w_1, \zeta_2, \dots, \zeta_n, z')
$$
satisfies the requirements of \eqref{Psi-properties}, and parametrizes $L$ locally. This is shown by adapting the argument of \cite[Proof of Proposition 6.2]{GHS2} in a straightforward way (which itself is a minor variation on \cite[Theorem 21.2.18]{Ho3}), so we omit the details. This establishes part (iii) of the Proposition. When working close to $x = x' = 0$, we need to use coordinates as in \cite[Proposition 7.5]{GHS2} and apply \cite[Lemma 7.6 and (i) of Lemma 7.7]{GHS2}, and we end up with the statement in part (ii).
\end{proof}

\begin{remark}\label{rem:dist-conic} The Lagrangian $L$ is smooth up to the boundary when viewed as a submanifold in the `scattering-fibred cotangent bundle' described in \cite{GHS1}. The boundary at $\bfc$ is naturally isomorphic to $L^{\bfc}$ in Proposition~\ref{prop:osc-form-low}. Correspondingly, we find that the distance function $d(z,z')$ on $M^2_b$ satisfies
$$
d(z,z') - d_{\conic}(y, y', \frac1{x}, \frac{\sigma}{x}) = e(z,z')
$$
is a bounded function on $M^2_b$, or more precisely on that part of $M^2_b$ where $x, x' \leq \eta$ and $d_{\partial M}(y, y') \leq \eta$ for sufficiently small $\eta$ (see \cite[Lemma 9.4]{HTW1}). From this we see that the results of Proposition~\ref{prop:osc-form-low} and Proposition~\ref{prop:osc-form-high} are compatible, as the factor $\exp{({i\lambda e(z,z')})}$ which is the discrepancy between \eqref{a} and \eqref{atilde}, and between \eqref{Phi-properties}(d) and \eqref{Psi-properties}(d), can be absorbed in the symbol $\tilde a$, respectively $b$.
\end{remark}

\begin{remark} The results of this paper could be extended to long range scattering metrics, as treated in \cite{HTW}. However, this would require an extension of the results of \cite{HV2}, \cite{HW} and \cite{GHS1} to Lagrangian submanifolds which are only conormal, rather than smooth, at the boundary. If this were done, then the discrepancy $e(z,z')$ between the distance function and the conic distance function is no longer smooth, or even bounded, but rather is conormal at the boundary with a bound of the form $(x+x')^{-1+\epsilon}$ at the boundary of $M^2_b$, i.e. a bit smaller than the distance functions themselves. In this case, the correct description of the localized spectral measure is with the true distance function $d(z,z')$ as phase function, rather than \eqref{a}, which is only true in the short range case.
\end{remark}


\section{Proof of Proposition~\ref{prop:localized spectral measure}}\label{sec:localized}

We now prove Proposition~\ref{prop:localized spectral measure}.
We define our partition of unity $Q_j$ by combining the low energy and high energy partitions. We choose a cutoff function $\chi(\lambda)$ supported in $[0,2]$ such that $1 - \chi$ is supported in $[1, \infty)$, and define
\begin{equation}\label{Qi}
\begin{split}
Q_1(\lambda)=\chi(\lambda) \big( Q_0^{\mathrm{low}}+Q_1^{\mathrm{low}} \big) + (1 - \chi(\lambda)) Q_1^{\mathrm{high}}, \\
Q_j(\lambda)=
\chi(\lambda)Q_j^{\mathrm{low}} + (1-\chi(\lambda))Q_j^{\mathrm{high}}, \quad \mathrm{for}~ 2\leq j\leq N_l;\\
Q_j(\lambda)=
(1-\chi(\lambda))Q_{j}^{\mathrm{high}}, \quad \mathrm{for}~ N_l+1\leq j\leq N.
\end{split}
\end{equation}

We first note that the term with $Q_1(\lambda)$ satisfies \eqref{beanQ} (with only the `$b$' term present) and \eqref{beans},
according to Proposition~\ref{prop:osc-form-low} and Proposition~\ref{prop:osc-form-high}. (In the case of low energies we also need to use Remark~\ref{rem:dist-conic} which tells us that we can replace the distance function by the conic distance function $d_{\conic}$ in \eqref{beanQ} without affecting the estimates on the amplitudes $a_{\pm}$.)
%
%

Next we prove the Proposition for  low energies, i.e. for $\lambda \leq 2$, and for $j \geq 2$.
Consider the second type of representation, \eqref{lsmoi},  in Proposition~\ref{prop:osc-form-low}.
%
We break the estimate into various cases. We first observe that
estimates of the form \eqref{bean} and \eqref{beans} are unaffected
by multiplication by a cutoff function of the form $\chi(\lambda
d(z,z'))$, where $\chi \in C_c^\infty(\RR)$. Therefore, we may treat
the cases that $\lambda d(z,z') \lesssim 1$ and $\lambda d(z,z')
\gtrsim 1$ separately. Consider first the case $\lambda d(z,z')
\lesssim 1$, or equivalently, $|w| \lesssim \rho$. In this case, we
show that the \eqref{lsmoi} has the form \eqref{beanQ} where only
the `$b$' term is present, satisfying \eqref{beans}. Thus, we need
to show that
$$
\big( \lambda \partial_\lambda)^\alpha \int_{\R^{n-1}}e^{
i\lambda\Phi(y,w,v)/x} \tilde a(\lambda,x/\lambda,y,w_1,v) \, dv
$$
is uniformly bounded.
For $\alpha = 0$ this is obvious. So consider the effect of applying  $\lambda
\partial_\lambda$. This is harmless when it
hits $\tilde a$. When it hits the phase it brings down a factor $i\lambda
\Phi/x$. We have $\lambda
\Phi/x = \Phi/\rho = v \cdot d_v \Phi/\rho + O(w_1/\rho)$, and
since $|w| \lesssim \rho$ the $O(w_1/\rho)$ is harmless. To treat the $v
\cdot d_v \Phi/\rho$ term, we can write using (b) of
\eqref{Phi-properties}
$$
\frac{v \cdot d_v \Phi}{\rho} e^{i\Phi/\rho} = -i v \cdot d_v e^{i\Phi/\rho},
$$
and integrating by parts we see that this term is $O(1)$ after integration. Repeated applications of $\lambda
\partial_\lambda$ are treated similarly.

Second, suppose that $|w| \geq C \rho$ for
some large $C$, but that $|w_1| \leq \rho$. For large enough $C$, this means that $d_{v_j} \Phi
\neq 0$, for some $j\geq2$, since by (a) of \eqref{Phi-properties},
we have $d_{v_j} \Phi = w_j - O(w_1)$. So by choosing $j$ so that
$|w_j|$ is maximal, and then $C$ large enough, we have $|d_{v_j}
\Phi| \geq c |w|$. Then we can write $$ e^{i\Phi/\rho} = \Big(
\frac{\rho d_{v_j} }{id_{v_j} \Phi} \Big)^N e^{i\Phi/\rho},
$$
and integrate by parts. Each integration by parts gains us a factor
of $\rho/|w|$. Thus we can estimate \eqref{lsmoi} by
$(1+|w|/\rho)^{-K} = (1+ \lambda d(z,z'))^{-K} $ for any $K$. Estimating the terms for $\alpha > 0$
is done just as in the first case above.

Third, suppose that $|w| \geq C |w_1|$
for some large $C$, and that $|w_1| \geq \rho$. Then we can integrate by parts and gain any
number of factors of $(1+ \lambda d(z,z'))^{-1}$ as in the second case above.

Finally we come to the case where $|w_1| \geq \rho$ and $|w_1|$ is
comparable to $|w|$. In this case,   we have removed a neighbourhood of $N^* \diagb$ from the microlocal support of the
localized spectral measure. As discussed in Section~\ref{sec:low}, in this region the Lagrangian $L^{\bfc}$ is a union of two sheets, each of which projects diffeomorphically to the base $\bfc$, and which are parametrized by the phase function $\pm d_{\conic}$ (in terms of the phase function $\Phi$ as in \eqref{lsmoi}, \eqref{Phi-properties}, this simply corresponds to the sign of $w_1$). We can thus split this case into two parts, according to the sign of $w_1$, and these give rise to the `$\pm$' terms in \eqref{beanQ}.

In this case, the key is to exploit property (c) of \eqref{Phi-properties}.  Define
\begin{equation}\label{3.27}
\begin{split}
\widetilde{\Phi}(x,y,w,v)=|w_1|^{-1}(\Phi(y,w,v) \mp xd(z,z')),
\end{split}
\end{equation}
and let $\omega=|w_1|/\rho$, then we need to estimate
\begin{equation*}
\begin{split}
\lambda^\alpha \partial_\lambda^\alpha
a(\lambda,z,z')&=\sum_{\beta+\gamma=\alpha}\frac{\alpha!}{\beta!\gamma!}\omega^\beta\int\limits_{\R^{n-1}}e^{
i\omega \widetilde{\Phi}(x,y,w,v)}\widetilde{\Phi}^{\beta}
\big(\lambda^\gamma \partial_\lambda^{\gamma}
 \tilde a \big)(\lambda,\rho,y,w_1,v)dv.
\end{split}
\end{equation*}
Let $\tilde{b}=\lambda^{\gamma}\partial_\lambda^{\gamma} \tilde a$, then
$|\partial_{\lambda}^\gamma \tilde{b}|\leq
C_{\gamma}\lambda^{-\gamma}$. Thus note $\omega\geq1$, it reduces to
show for any $0\leq\beta\leq \alpha$
\begin{equation}\label{3.28}
\begin{split}
\Big|\int_{\R^{n-1}}e^{ i\omega
\widetilde{\Phi}(x,y,w,v)}(\omega\widetilde{\Phi})^{\beta}
\tilde{b}(\lambda,\rho,y,w_1,v)dv\Big|\leq C \omega^{-\frac{n-1}2}.
\end{split}
\end{equation}
To proceed, we fix $(x, y, w)$ with $w \neq 0$ (and hence $w_1 \neq 0$ due to our assumption that $|w_1|$ is comparable to $|w|$). We use a cutoff function $\Upsilon$ to divide the $v$ integral into two parts: the support of $\Upsilon$, in which $|d_v \widetilde\Phi| \geq \tilde\epsilon/2,$ and the other on the support of $1 - \Upsilon$, in which $|d_v \widetilde\Phi| \leq \tilde\epsilon$. On the support of $\Upsilon$, we integrate by parts in $v$ and gain any power of $\omega^{-1}$, proving \eqref{3.28}. On the support of $1 - \Upsilon$, we
make the variable change
$$(v_2\cdots,v_n)\rightarrow (\theta_2,\cdots,\theta_{n}),\quad \theta_i=d_{v_i}\widetilde{\Phi},~i=2\cdots, n.$$
Note that by property (c) of \eqref{Phi-properties},
$$
\frac{\partial \theta_j}{\partial v_k} = d^2_{v_jv_k} \widetilde \Phi = \pm A_{jk}.
$$
The nondegeneracy of $A_{jk}$ shows that this change of variables is locally nonsingular, provided $\tilde\epsilon$ is sufficiently small. Thus, for each point $v$ in the support of $1 - \Upsilon$, there is a neighbourhood in which we can change variables to $\theta$ as above. Using the compactness of the support of $b$ in \eqref{lsmoi}, we see that there are a finite number of neighbourhoods covering the intersection of the support of $\Upsilon$ and the $v$-support of $b$. For simplicity of exposition, we assume that there is only one such neighbourhood $U$ below.

Denote $\mathcal{B}_\delta:=\big\{\theta:|\theta|\leq \delta\big\}$, and
choose a  $C^\infty$ function $\chi_{\mathcal{B}_\delta}(\theta)$ which equals
$1$ when on the set $\mathcal{B}_\delta$ but equals $0$ for outside
$\mathcal{B}_{2\delta}$, and with derivatives bounded by
$$
\big| \nabla^{(j)}_\theta \chi_{\mathcal{B}_\delta}(\theta) \big| \leq C \delta^{-j}.
$$
Here $\delta$ is a parameter to be chosen later (depending on $\omega$).
Consider the integral \eqref{3.28} after changing variables and with the cutoff function $\chi_{\mathcal{B}_\delta}(\theta)$ inserted (note that $1 - \Upsilon = 1$ on the support of $\chi_{\mathcal{B}_\delta}(\theta)$, provided $\delta \leq \tilde\epsilon/2$):
\begin{equation*}
\begin{split}
\Big|\int e^{i\omega
\widetilde{\Phi}(x,y,w,\theta)}\big(\omega\widetilde{\Phi}\big)^{\beta}
\tilde{b}(\lambda,\rho,y,w_1,\theta)\chi_{\mathcal{B}_\delta}(\theta)\frac{d\theta}{|A^{-1}(y,w,\theta)|}\Big|.
\end{split}
\end{equation*}
Using property (d) of \eqref{Phi-properties}, we see that
$\widetilde \Phi = 0$  when $\theta = 0$.  Also, due to our choice
of $\theta$, we have  $d_\theta \widetilde \Phi = 0$ when $\theta =
0$. Hence $\widetilde \Phi = O(|\theta|^2)$. Hence
\begin{equation*}
\begin{split}
\Big|\omega^\beta\int e^{ i\omega
\widetilde{\Phi}(x,y,w,\theta)}\widetilde{\Phi}^{\beta}
\tilde{b}(\lambda,\rho,y,w_1,\theta)\chi_{\mathcal{B}_\delta}(\theta)
\frac{d\theta}{|A^{-1}(y,w,\theta)|}\Big|\leq
C(\omega\delta^2)^\beta\delta^{n-1}.
\end{split}
\end{equation*}

It remains to treat the integral with cutoff
$(1-\chi_{\mathcal{B}_\delta}({\theta}))$ inserted. Notice that
$|d_\theta \widetilde \Phi|$ is comparable to $|\theta|$ since
$d_\theta \widetilde \Phi = 0$ when $\theta = 0$, and
$$
d^2_{\theta_i \theta_j} \widetilde \Phi = \sum_{k,l} (A^{-1})_{il}(A^{-1})_{jk} d^2_{v_kv_l}\widetilde\Phi
$$
is nondegenerate when $\theta = 0$.
We define the
 differential operator $L$ by
\begin{equation*}
\begin{split}
L=\frac{-id_{\theta}\widetilde{\Phi}\cdot\partial_{\theta}}{\omega \big|d_{\theta}\widetilde{\Phi}\big|^2}.
\end{split}
\end{equation*}
Then the adjoint operator is given by
\begin{equation*}
\begin{split}
\leftidx{^{t}}L=-L+\frac{i}{\omega}\Big(\frac{\Delta_{\theta}
\widetilde{\Phi}}{|d_{\theta}\widetilde{\Phi}|^2}-2\frac{d^2_{{\theta}_j{\theta}_k}
\widetilde{\Phi} \, d_{{\theta}_j}\widetilde{\Phi} \, d_{{\theta}_k}\widetilde{\Phi}}{|d_{\theta}\widetilde{\Phi}|^4}\Big).
\end{split}
\end{equation*}
Since $L e^{i\omega \widetilde \Phi} = e^{i\omega \widetilde \Phi}$, we integrate by parts $N$ times to obtain
\begin{equation*}
\begin{split}
\Big|\int e^{ i
\omega\widetilde{\Phi}(x,y,w,{\theta})}&(\omega\widetilde{\Phi})^{\beta}
\tilde{b}(\lambda,\rho,y,w_1,{\theta})(1-\chi_{\mathcal{B}_\delta}({\theta}))(1 - \Upsilon) \, d{\theta}\Big|\\&\leq
C \int \Big|(\leftidx{^{t}}L)^N\big((\omega
\widetilde{\Phi})^{\beta}
\tilde{b}(\lambda,\rho,y,w_1,{\theta})(1-\chi_{\mathcal{B}_\delta}({\theta}))(1 - \Upsilon)\big)
\Big|d{\theta}.
\end{split}
\end{equation*}
Inductively we find that
$$\big|(\leftidx{^{t}}L)^N\big((\omega\widetilde{\Phi})^{\beta}
\tilde{b}(1-\chi_{\mathcal{B_\delta}})(1 - \Upsilon)\big)\big|\leq C
\omega^{-N+\beta}\max\big\{|\theta|^{2\beta-2N},
|\theta|^{2\beta-N}\delta^{-N}\big\}.$$
Choosing $N$ large enough,
we get
\begin{equation*}
\begin{split}
\Big|\int e^{ i \omega
\widetilde{\Phi}(x,y,w,{\theta})}&(\omega\widetilde{\Phi})^{\beta}
\tilde{b}(\lambda,\rho,y,w_1,{\theta})
(1-\chi_{\mathcal{B_\delta}})(1 - \Upsilon) \, d{\theta}\Big|\\&\leq
\omega^{-N+\beta}\int_{|\theta|\geq\delta}\big(|\theta|^{2\beta-2N}+
|\theta|^{2\beta-N}\delta^{-N}\big)d\theta\leq C
\omega^{-N+\beta}\delta^{2\beta-2N}\delta^{n-1}.
\end{split}
\end{equation*}
Choose $\delta=\omega^{-1/2}$ to balance the two parts of the integral (with $\chi_{\mathcal{B_\delta}}$ and with $1-\chi_{\mathcal{B_\delta}}$). We finally
obtain
\begin{equation*}
\begin{split}
\Big|\int e^{ i
\omega\widetilde{\Phi}(x,y,w,{\theta})}&(\omega\widetilde{\Phi})^{\beta}
\tilde{b}(\lambda,\rho,y,w_1,{\theta}) (1 - \Upsilon) \,  d{\theta}\Big|\leq C
\omega^{-(n-1)/2},
\end{split}
\end{equation*}
which proves \eqref{3.28} as desired.\vspace{0.2cm}

We next sketch how to prove \eqref{beans} in the high energy case, $i > N_l$. In terms of
Proposition~\ref{prop:osc-form-high}, consider a term of type (iii); it suffices to show
\begin{equation*}
\begin{split}
a(h,z,z') =e^{\mp i d(z,z')/h}\int_{\R^{n-1}}e^{
i\Psi(y,w,x,v)/(xh)}b(h,x,y,w_1,v)dv,
\end{split}
\end{equation*}
satisfies
\begin{equation*}
\Big|(h \partial_h)^\alpha a(h,z,z') \Big|\leq C_\alpha \big(1+
\frac{|w|}{xh}\big)^{-\frac{n-1}2}.
\end{equation*}
Notice that $\lambda=1/h$ and $\Psi$ has the same properties (a) --- (d) as
$\Phi$. Therefore the low energy  proof works verbatim, with the argument $x$ of $\Psi$ acting as a smooth parameter, and leads to the
desired conclusion. The proof in case (ii) works in exactly the same way, with $w$ given by $z-z'$.

\begin{remark} To illustrate this theorem, consider the case of the spectral measure on flat $\RR^3$, which is
$$
dE_{\sqrt{\Delta}}(\lambda)(z,z') = \frac1{2\pi^2} \frac{\lambda^2 \sin \lambda |z-z'|}{\lambda |z-z'|} d\lambda.
$$
We decompose this, using the cutoff function $\chi$ as in
\eqref{Qi},  according to the size of $\lambda|z-z'|$. Where
$\lambda|z-z'| \geq 1$, that is, more than one wavelength from the
diagonal, we split the sine factor into exponential terms. Within
$O(1)$ wavelengths of the diagonal, however, we keep the sine factor
as is, to exploit the cancellation in the difference
$e^{+i\lambda|z-z'|} - e^{-i\lambda|z-z'|}$ when $\lambda |z-z'|$ is
small. This gives as an expression
$$
\frac{\lambda^2}{2\pi^2} \Big( (1 - \chi)(\lambda |z-z'|) \frac{e^{i\lambda |z-z'|}}{2i\lambda |z-z'|} - (1 - \chi)(\lambda |z-z'|) \frac{e^{-i\lambda |z-z'|}}{2i\lambda |z-z'|} +
\chi(\lambda |z-z'|) \frac{\sin \lambda |z-z'|}{\lambda |z-z'|}  \Big).
$$
This is a decomposition into `$\pm$' and `$b$' terms as in \eqref{beanQ}, where the amplitudes satisfy \eqref{bean} and \eqref{beans}. So we can think of the $b$ term as the near-diagonal term, and the other terms as related to the two sheets of the Lagrangian $L$ or $L^{\bfc}$ which are separated away from the diagonal. The function of the microlocalizing operators $Q_j(\lambda)$ (which are not required in the case of flat Euclidean space) is to remove parts of the Lagrangian which do not project diffeomorphically to the base.
\end{remark}


\section{$L^2$ estimates}\label{sec:L2}
In this section, we prove $L^2 \to L^2$ estimates on microlocalized
versions of the Schr\"odinger propagator, using the operator
partition of unity $Q_j$ described at the beginning of the previous section, based on  \cite{GHS2}.\vspace{0.2cm}

We begin by defining microlocalized propagators. First we give a formal definition.
It is not immediately clear that the formal definition is well-defined, so our first task is to show this. We do so by showing that each microlocalized propagator is a bounded operator on $L^2$. This serves both to show the well-definedness of each microlocalized propagator, and to establish the $L^2 \to L^2$ estimate needed for the abstract Keel-Tao argument.

We define, as in the Introduction,
\begin{equation}\label{Uit}
U_j(t) = \int_0^\infty e^{it\lambda^2} Q_j(\lambda)
dE_{\sqrt{\mathbf{H}}}(\lambda)
\end{equation}
where $Q_j$ is the decomposition defined in \eqref{Qi}.

Our first task is to make sense of this expression. We do this by showing that each $U_j(t)$ is a bounded operator on $L^2(M^\circ)$. We have

\begin{proposition}\label{prop:L2}
The integral \eqref{Uit}  defining
$U_j(t)$ are well-defined on each finite interval, and converge on
$\RR_+$ in the strong operator topology to define bounded operators
on $L^2(M^\circ)$. Moreover, the operator norm of $U_j(t)$ on $L^2(M^\circ)$ are
bounded uniformly for $t \in \RR$. Finally, we have
\begin{equation}
\sum_j U_j(t) = e^{it\mathbf{H}}.
\label{sumUi}\end{equation}
\end{proposition}

The rest of this section is devoted to proving this Proposition.

Suppose that $A(\lambda)$ is a family of bounded operators on
$L^2(M^\circ)$, compactly supported and $C^1$ in $\lambda \in (0,
\infty)$. Integrating by parts, the integral of
$$
\int_{0}^\infty A(\lambda) \specl
$$
is  given by
$$
- \int_0^\infty \Big( \frac{d}{d\lambda} A(\lambda) \Big) \El \, d\lambda.
$$
In view of Corollaries~\ref{cor:bdd-le} and \ref{cor:bdd-he}, we can take $A(\lambda)$ to be a smooth function of $\lambda$ with compact support in $(0, \infty)$ times  $e^{it\lambda^2}
Q_j(\lambda)$. This means that the integral
\eqref{Uit} is well-defined over any
compact interval in $(0, \infty)$. We need to show that the integral
over the whole of $\RR_+$ converges in the strong operator topology.
To do so,  we introduce a dyadic partition of unity on the
positive $\lambda$ axis by choosing $\phi \in C_c^\infty([1/2, 2])$,
taking values in $[0,1]$, such that
$$
\sum_{m \in \ZZ} \phi\big( \frac{\lambda}{2^m} \big) = 1.
$$
We now define
\begin{equation}\begin{gathered}
U_{j,m}(t) = -\int_0^\infty   \frac{d}{d\lambda} \Big(e^{it\lambda^2} \phi\big( \frac{\lambda}{2^m} \big) Q_j(\lambda) \Big) \El  .
\end{gathered}\label{Ui-sum-defn}\end{equation}

We next show that the sum over $m$ of the operators $U_{j,m}(t)$ in
\eqref{Ui-sum-defn} is well-defined. For this we use the
Cotlar-Stein lemma, which we recall here (we use the version in
\cite[Chapter 8]{Gr}):

\begin{lemma}[Cotlar-Stein lemma]\label{Cotlar} Suppose that $A_j$ are a sequence of bounded linear operators on a Hilbert space $H$ such that
\begin{equation}\label{4.8}
\|A_m^*A_n\|_{H \to H}\leq \big(\gamma(m-n)\big)^2,\quad
\|A_mA_n^*\|_{H \to H}\leq \big(\gamma(m-n)\big)^2,\end{equation}
where $\{\gamma(m)\}_{m\in \Z}$ is a sequence of positive constants
such that $C=\sum_{m\in\Z}\gamma(m)<\infty$. Then for all $f \in H$,
the sequence $\sum_{|m| \leq N} A_m f$ converges as $N \to \infty$
to an element $Af \in H$. The operators $A = \sum_m A_m$ and $A^* =
\sum_m A^*_m$ so defined (in the strong operator topology)
 satisfy
\begin{equation}\label{4.9}
\|A\|_{H\rightarrow H}\leq C, \quad \|A^*\|_{H\rightarrow H}\leq C.\end{equation}
Moreover, the operator norms of $\sum_{m \in J} A_m$ and $\sum_{m \in J} A_m^*$  are  bounded by $C$ for any finite subset $J$ of the integers.
\end{lemma}

We also use the following Lemma:

\begin{lemma}\label{BB*} Suppose that for $l = 1, 2$,   $A_l(\lambda)$ is a family of operators compactly supported in $\lambda$ in the open interval $(0, \infty)$, and with $A_l(\lambda)$, $\partial_\lambda A_l(\lambda)$ uniformly bounded on $L^2(M^\circ)$. Define
$$
B_l =  \int A_l(\lambda) \specl .
$$
Then
$$
B_1 B_2^* = \int A_1(\lambda) \specl A_2(\lambda)^* ,
$$
where by definition the last expression is equal to
\begin{equation}
\int \Big( - \frac{d}{d\lambda}A_1(\lambda) \Big) \El A_2(\lambda) - A_1(\lambda) \El \Big( \frac{d}{d\lambda}A_2(\lambda) \Big) .
\label{AA*}\end{equation}
\end{lemma}

\begin{proof} We compute
\begin{equation}\begin{gathered}
B_1 B_2^* = \int \int \Big( \frac{d}{d\lambda}A_1(\lambda) \Big) \El \Em \Big( \frac{d}{d\mu}A_2(\mu)^* \Big) \, d\lambda \, d\mu \\
= \iint_{\lambda \leq \mu} \Big( \frac{d}{d\lambda}A_1(\lambda) \Big) \El \Big( \frac{d}{d\mu}A_2(\mu)^* \Big) \, d\lambda \, d\mu \\ + \iint_{\mu \leq \lambda} \Big( \frac{d}{d\lambda}A_1(\lambda) \Big)  \Em \Big( \frac{d}{d\mu}A_2(\mu)^* \Big) \, d\lambda \, d\mu \\
= \int  \Big( \frac{d}{d\lambda}A_1(\lambda) \Big) \El \big(-A_2(\lambda)^*\big) \, d\lambda   + \int \big(-A_1(\mu) \big)  \Em \Big( \frac{d}{d\mu}A_2(\mu)^* \Big)  \, d\mu \\
= \eqref{AA*}.
\end{gathered}\end{equation}

\end{proof}
Now we show that the sum in \eqref{Ui-sum-defn} is well-defined. We first note a simplification: since the $Q_j(\lambda)$ are a partition of the identity, we have
$$
V_m(t) := \sum_{j=1}^{N} U_{j,m}(t) = \int e^{it\lambda^2} \chi(\lambda) \phi\big( \frac{\lambda}{2^m} \big) \specl,
$$
which is clearly bounded on $L^2(M^\circ)$ with operator norm $\leq 1$ using spectral theory.
Moreover, the sum of any subset of the $V_m$ converges strongly to an operator with norm $\leq 1$.
Due to this, we may ignore the case $j=1$ and prove the $L^2$-boundedness only
for $j \geq 2$.

We have, by Lemma~\ref{BB*},
\begin{equation}\begin{gathered}
U_{j,m}(t) U_{j,n}(t)^* =
\int   \chi(\lambda)^2 \phi\big( \frac{\lambda}{2^m} \big) \phi\big( \frac{\lambda}{2^n} \big)
Q_j(\lambda) \specl Q_j(\lambda)^* \\
= -\int \frac{d}{d\lambda} \Big( \chi(\lambda)^2 \phi\big( \frac{\lambda}{2^m} \big) \phi\big( \frac{\lambda}{2^n} \big)
Q_j(\lambda) \Big) \El Q_j(\lambda)^*  \\
- \int  \chi(\lambda)^2 \phi\big( \frac{\lambda}{2^m} \big) \phi\big( \frac{\lambda}{2^n} \big)
Q_j(\lambda) \El \frac{d}{d\lambda} Q_j(\lambda)^*.
\end{gathered}\label{Uijk}\end{equation}
We observe that this is independent of $t$, and is identically zero unless $|m-n| \leq 2$.
When $|m-n| \leq 2$, we note that the integrand is a bounded operator on $L^2$, with an operator bound of the form $C/\lambda$ where $C$ is uniform, as we see from Corollary~\ref{cor:bdd-le} and the support property of $\phi$. The integral is therefore
uniformly bounded, as we are integrating over a dyadic interval in $\lambda$.

We next consider the operators $U_{j,m}^*(0) U_{j,n}(0)$, just in the case $t = 0$. This has an expression
$$\begin{gathered}
 \iint \El \frac{d}{d\lambda} \Big( \phi\big( \frac{\lambda}{2^m} \big)
Q_j(\lambda)^* \Big) \frac{d}{d\mu} \Big( Q_j(\mu) \phi\big( \frac{\mu}{2^n} \big) \Big) \Em \, d\lambda \, d\mu.
\end{gathered}$$
It is clear that each of these operators is uniformly bounded in $m , n$ in operator norm. To apply Cotlar-Stein, we show a estimate of the form $C 2^{-|m-n|}$ for the operator norm of this term.
Write $Q^*_{j,m}(\lambda), Q_{j,n}(\mu)$ for the operators in
parentheses above.
Consider first the case, $2 \leq j \leq N_l$, in which case $Q_j$ has Schwartz kernel supported near the boundary of the diagonal.
For convenience of exposition, we assume that $\lambda, \mu \leq 2$ (or equivalently, $m, n \leq 1$). Then by the construction of $Q_j$, $2 \leq j \leq N_l$ (see  Section~\ref{subsec:lepoi} and \eqref{Qi}),  the scattering pseudodifferential operators
$Q^*_{j,m}(\lambda), Q_{j,n}(\mu)$ are smooth and compactly
supported in $x'/\lambda,x'/\mu$ respectively and are microlocally
supported near the characteristic set. More precisely, we see the
composition of the two scattering pseudodifferential operators for
$j \geq 2$ takes the form
\begin{equation*}
\begin{split}
&Q^*_{j,m}(\lambda)Q_{j,n}(\mu)\\&=\int
e^{-i\lambda\big((y-y')\cdot\eta+(\sigma-1)\nu\big)/x'}e^{i\mu\big((y'-y'')\cdot\eta'+(\sigma'-1)\nu'\big)/x'}
\\&\qquad\quad \times q_{j,m}(\lambda,y',\frac{x'}{\lambda},\eta,\nu)q_{j,n}(\mu,y',\frac{x'}{\mu},\eta',\nu')dx'dy'd\eta
d\nu d\eta' d\nu'
\end{split}
\end{equation*}
where $\sigma=x'/x, \sigma'=x'/x''$, and $q_{j,m}, q_{j,n}$ are
smooth and polyhomogeneous in $\lambda,\mu$ and  compactly
supported in $x'/\lambda,x'/\mu, y'$. In addition, we have  $\nu^2+|\eta|^2\geq 1/4$ and
$\nu'^2+|\eta'|^2\geq 1/4$ on the support of $q_{j,m}q_{j,n}$. By symmetry, we assume $\lambda>\mu$
without loss of generality. Let us introduce the operator
$$\mathcal{L}=i[\lambda(|\nu|^2+|\eta|^2)]^{-1}(x'\eta \partial_{y'}-\nu
x'^2\partial_{x'}),$$ then $\mathcal{L}
e^{-i\lambda\big((y-y')\cdot\eta+(\sigma-1)\nu\big)/x'}=
e^{-i\lambda\big((y-y')\cdot\eta+(\sigma-1)\nu\big)/x'}$. By using
$\mathcal{L}$ to integrate by parts, we gain the factor
$\lambda^{-1}$ since $|\nu|^2+|\eta|^2$ is uniformly bounded from
below; we incur a factor $\mu$ if the derivative falls on
$e^{i\mu\big((y'-y'')\cdot\eta'+(\sigma'-1)\nu'\big)/x'}$, or a
factor of $x'$ or ${x'}^2/\mu$ if the derivative falls on $q_{j,m}$
or $q_{j,n}$. Since $x' \leq \mu$ on the support of $q_{j,m}$, we
have an overall gain of $\mu/\lambda\sim 2^{-|m-n|}$. The
$L^2$-boundness of the spectral projection gives
$\|U^*_{j,m}(0)U_{j,n}(0)\|_{L^2\rightarrow L^2}\leq
C2^{-|m-n|}$.\vspace{0.2cm} A similar argument works if one or both of $m, n$ are $\geq 1$.

A similar estimate is true in the case $N_l + 1 \leq j \leq N$, in which case we are automatically in the high energy case, and with Schwartz kernels supported in the interior of $M^\circ \times M^\circ$.  The argument is also almost exactly the same as the previous case.  We  can write the composition
$$
 \frac{d}{d\lambda} \Big(  \phi\big( \frac{\lambda}{2^j} \big)
Q_j(\lambda)^* \Big)
\frac{d}{d\mu} \Big(Q_j(\mu) \phi\big( \frac{\mu}{2^k} \big)\Big)
$$
in the form
\begin{equation}
\lambda^n \mu^n \iiint e^{i\lambda(z - z'') \cdot \zeta}  q_{j,m}(z'', \zeta, \lambda) e^{i\mu (z'' - z') \cdot \zeta'}  q_{j,n}(z'', \zeta', \mu) \, d\zeta \, d\zeta' \, dz''
\label{he}\end{equation}
where $q_{j,m}$ is supported where $\lambda \sim 2^m$, $|\zeta|^2 \sim 1$ and is such that  $D_z^\alpha D_\zeta^\beta q_{j,m}$ is bounded by $C\lambda^{-1}$.
Assume without loss of generality that $m>n$, i.e. $\lambda > \mu$ on the support of the integrand.
We note that the differential operator
$$\mathcal{L}= \frac{i\zeta \cdot \partial_{z''}}{\lambda |\zeta|^2}$$
leaves $e^{i\lambda(z - z'') \cdot \zeta}$ invariant, so we can apply it to this phase factor in the integral \eqref{he}.  Integrating by parts, the $\partial_{z''}$ derivative either hits the other phase factor $e^{i\mu (z'' - z') \cdot \zeta'}$, in which case we incur a factor of $\mu$, or it hits one of the symbols $q_{i,j}$ or $q_{i, k}$, in which case we incur no factor. So we gain a factor of either $\mu/\lambda \sim 2^{-|j-k|}$, or $1/\lambda$ which is even better since $\mu > 1$ on the support of $q_{j,n}(z'', \zeta', \mu)$. This completes the Cotlar-Stein estimates for $U_i(0)$.

It now follows from the Cotlar-Stein Lemma that $U_j(0)^*$, $j = 2
\dots N$, is well defined as the strong limit of the sequence of
operators
$$
\sum_{|m| \leq l} U_{j,m}(0)^*.
$$
Consider the sequence $\sum_{|m|\leq l} U_{j,m}(t)^*$. We claim that
this sequence converges strongly, and  define $U_j(t)^*$ to be this
limit. To prove this claim, choose an arbitrary $f \in
L^2(M^\circ)$. We have shown that
$$ \lim_{l \to \infty} \  \sup_{L
> l} \big\| \sum_{l \leq |m| \leq L} U_{j,m}(0)^* f \big\|_2^2 =
0.
$$
This is equivalent to
$$
\lim_{l \to \infty} \ \sup_{L > l} \sum_{l \leq |m|, |m'| \leq L} \  \langle U_{j,m}(0) U_{j,m'}(0)^* f , f \rangle  = 0.
$$
But we saw in \eqref{Uijk} that $U_{j,m}(0) U_{j,m'}(0)^* = U_{j,m}(t) U_{j,m'}(t)^*$. Hence we have
$$
\lim_{l \to \infty} \  \sup_{L > l} \sum_{l \leq |m|, |m'| \leq L} \ \langle U_{j,m}(t) U_{j,m'}(t)^* f , f \rangle  = 0,
$$
which implies that
$$
\lim_{l \to \infty} \  \sup_{L > l}   \big\| \sum_{l \leq |m| \leq L} U_{j,m}(t)^* f \big\|_2^2 =  0.
$$
Hence the sequence $\sum_{|m| \leq l} U_{j,m}(t)^* f$ converges for every $f \in L^2(M^\circ)$ as $l \to \infty$, i.e.~ the sequence $\sum_{|m| \leq l} U_{j,m}(t)^*$ converges strongly. We see from this that the integral
$$
\int e^{-it\lambda^2} \specl Q_j(\lambda)^*
$$
converges in the strong topology, hence defines $U_j(t)^*$.
Finally we show that the operator norm of $U_j(t)^*$ is bounded uniformly in $t$. Since $\sum_{|m| \leq l} U_{j,m}(t)^*$ converges in the strong operator topology, we have
$$
\| U_j(t)^* \|  \leq  \sup_{l \to \infty} \big\| \sum_{|m| \leq l} U_{j,m}(t)^* \big\| .
$$
But we have
$$\begin{gathered}
 \big\| \sum_{|m| \leq l} U_{j,m}(t)^* \big\|^2 =
 \Big\| \sum_{|m|, |m'| \leq l} \!\! U_{j,m}(t) U_{j,m'}(t)^* \Big\|
 =  \Big\| \sum_{|m|, |m'| \leq l} \!\! U_{j,m}(0) U_{j,m'}(0)^* \Big\|
\\  =  \big\| \sum_{|m| \leq l} U_{j,m}(0)^* \big\|^2
 \end{gathered}$$
 and the operator norm of $\sum_{|m| \leq l} U_{j,m}(0)^*$ is bounded uniformly in $l$ by the estimates proved above using the Cotlar-Stein Lemma.

\

This completes the proof of Proposition~\ref{prop:L2}. \vspace{0.2cm}

\begin{remark} This argument allows  us to avoid using a Littlewood-Paley type decomposition in this setting.
 Littlewood-Paley type estimates were established in \cite{Bouclet} for asymptotically  conic manifolds in the form of
\begin{equation*}
\|f\|_{L^p}\lesssim
\Big(\sum_{k\geq0}\|\phi(2^{-2k}\Delta_g)f\|^2_{L^p}\Big)^{\frac12}+\|\sum_{k\leq0}\phi(2^{-2k}\Delta_g)f\|_{L^p}.
\end{equation*}
\end{remark}


\section{Dispersive estimates}\label{sec:dispersive}
In this section, we use stationary phase and Proposition
\ref{prop:localized spectral measure} to establish the
microlocalized dispersive estimates.

\begin{proposition}[Microlocalized dispersive estimates]\label{dispersive}
Let $Q_j(\lambda)$ be defined in \eqref{Qi}. Then for all integers
$j\geq1$, the kernel estimate
\begin{equation}\label{disper}
\Big|\int_0^\infty e^{it\lambda^2} \big(Q_j(\lambda)
dE_{\sqrt{\mathbf{H}}}(\lambda)Q_j^*(\lambda)\big)(z,z')
d\lambda\Big|\leq C |t|^{-\frac{n}2}
\end{equation}
holds for a constant $C$ independent of points $z,z'\in M^\circ$.
\end{proposition}
\begin{proof} The key to the proof is to use the estimates in Proposition
\ref{prop:localized spectral measure}. We first consider $j=1$.
Since the term with $Q_1(\lambda)$ satisfies \eqref{beanQ} with only
the `$b$' term, then we can use the estimate \eqref{beans} to obtain
\begin{equation}\label{4.2}
\begin{split}
\Big|\big(\frac{d}{d\lambda}\big)^{N}\big(Q_1(\lambda)dE_{\sqrt{\mathbf{H}}}(\lambda)Q_1^*(\lambda)\big)(z,z')\Big|\leq
C_N\lambda^{n-1-N}\quad \forall N\in\mathbb{N}.
\end{split}
\end{equation}
Let $\delta$ be a small constant to be chosen later. Recall that we chose $\phi\in
C_c^\infty([\frac12,2])$ such that $\sum_{m \in \ZZ}\phi(2^{-m}\lambda)=1$; we
denote $\phi_0(\lambda)=\sum_{m\leq -1}\phi(2^{-m}\lambda)$. Then
\begin{equation*}
\Big|\int_0^\infty e^{it\lambda^2} \big(Q_1(\lambda)
dE_{\sqrt{\mathbf{H}}}(\lambda)Q_1^*(\lambda)\big)(z,z')
\phi_0(\frac{\lambda}{\delta})d\lambda\Big|\leq
C\int_0^\delta\lambda^{n-1}d\lambda\leq C\delta^n.
\end{equation*}
 We use  integration by parts $N$ times to obtain, using  \eqref{4.2},
\begin{equation*}
\begin{split}
&\Big|\int_0^\infty e^{it\lambda^2}
\sum_{m\geq0}\phi(\frac{\lambda}{2^m\delta})\big(Q_1(\lambda)
dE_{\sqrt{\mathbf{H}}}(\lambda)Q_1^*(\lambda)\big)(z,z')
d\lambda\Big|\\
&\leq \sum_{m\geq 0}\Big|\int_0^\infty
\big(\frac1{2\lambda
t}\frac\partial{\partial\lambda}\big)^{N}\big(e^{it\lambda^2}\big)
\phi(\frac{\lambda}{2^m\delta})\big(Q_1(\lambda)
dE_{\sqrt{\mathbf{H}}}(\lambda)Q_1^*(\lambda)\big)(z,z')
d\lambda\Big|\\& \leq
C_N|t|^{-N}\sum_{m\geq0}\int_{2^{m-1}\delta}^{2^{m+1}\delta}\lambda^{n-1-2N}d\lambda\leq
C_N|t|^{-N}\delta^{n-2N}.
\end{split}
\end{equation*}
Choosing $\delta=|t|^{-\frac12}$, we have thus proved
\begin{equation}\label{4.3}
\begin{split}
&\Big|\int_0^\infty e^{it\lambda^2} \big(Q_1(\lambda)
dE_{\sqrt{\mathbf{H}}}(\lambda)Q_1^*(\lambda)\big)(z,z')
d\lambda\Big|\leq C_N|t|^{-\frac n2}.
\end{split}
\end{equation}

Now we consider the case $j\geq2$. Let $r=d(z,z')$
and $\bar{r}=rt^{-\frac12}$. In this case, we write the kernel using
Proposition \ref{prop:localized spectral measure}
\begin{equation}\label{4.4}
\begin{split}
&\int_0^\infty e^{it\lambda^2} \big(Q_j(\lambda)
dE_{\sqrt{\mathbf{H}}}(\lambda)Q_j^*(\lambda)\big)(z,z')
d\lambda\\
&=\sum_\pm \int_0^\infty e^{it\lambda^2}e^{\pm
ir\lambda}\lambda^{n-1}a_\pm(\lambda,z,z')d\lambda+\int_0^\infty
e^{it\lambda^2}\lambda^{n-1}b(\lambda,z,z')d\lambda  \\&= t^{-\frac
n2}\sum_{\pm} \int_0^\infty e^{i\lambda^2}e^{\pm
i\bar{r}\lambda}\lambda^{n-1}a_\pm(t^{-1/2}\lambda,z,z')d\lambda
+\int_0^\infty e^{it\lambda^2}\lambda^{n-1}b(\lambda,z,z')d\lambda,
\end{split}
\end{equation}
where $a_\pm$ satisfies estimates
\begin{equation*}
\big|\partial_\lambda^\alpha a_\pm(\lambda,z,z') \big|\leq C_\alpha
\lambda^{-\alpha}(1+\lambda d(z,z'))^{-\frac{n-1}2},
\end{equation*}
and therefore
\begin{equation}\label{beans0}
\Big|\partial_\lambda^\alpha \big(a_\pm(t^{-1/2}\lambda,z,z')\big)
\Big|\leq C_\alpha \lambda^{-\alpha}(1+\lambda
\bar{r})^{-\frac{n-1}2}.
\end{equation}
By (1.16), the above term with $b(\lambda,z,z')$ can be estimated by
using the same argument as for $Q_1$. Now we consider first term in RHS
of \eqref{4.4}. We divide it into two pieces using the partition of
unity above. It suffices to prove that there exists a constant $C$
independent of $\bar{r}$ such that
\begin{equation*}
\begin{split}
I^\pm:=&\Big|\int_0^\infty e^{i\lambda^2}e^{\pm
i\bar{r}\lambda}\lambda^{n-1}a_\pm(t^{-1/2}\lambda,z,z')\phi_0(\lambda)d\lambda\Big|\leq
C,\\II^\pm:=& \Big|\sum_{m\geq0}\int_0^\infty e^{i\lambda^2}e^{\pm
i\bar{r}\lambda}\lambda^{n-1}a_\pm(t^{-1/2}\lambda,z,z')\phi(\frac{\lambda}{2^m})d\lambda\Big|\leq
C.
\end{split}
\end{equation*}
The estimate for $I^\pm$ is obvious, since $\lambda\leq 1$.
For $II^{+}$, we use integration by parts. Notice that
$$
L^+ (e^{i\lambda^2 + i \bar{r} \lambda}) =  e^{i\lambda^2 + i \bar{r} \lambda}, \quad L^+ = \frac{-i}{2\lambda + \bar{r}} \frac{\partial}{\partial \lambda}.
$$
Writing
$$
e^{i\lambda^2 + i \bar{r} \lambda} = (L^+)^N (e^{i\lambda^2 + i \bar{r} \lambda})
$$
and integrating by parts, we gain a factor of $\lambda^{-2N}$ thanks to \eqref{beans0}. Thus $II^{+}$ can be estimated by
$$
\sum_{m \geq 0} \int_{\lambda \sim 2^m} \lambda^{n-1-2N} \, d\lambda \leq C.
$$

To treat $II^-$, we introduce a further decomposition, based on the size of $\bar{r} \lambda$. We write
$II^-=II^-_1+II^-_2$, where (dropping the $-$ superscripts and subscripts from here on)
\begin{equation*}
\begin{split}
II_1=&\Big|\sum_{m\geq0}\int e^{i\lambda^2}e^{-
i\bar{r}\lambda}\lambda^{n-1}a(t^{-1/2}\lambda,z,z')\phi(\frac{\lambda}{2^m})
\phi_0(4\bar{r} \lambda) d\lambda\Big|, ~\\II_2=&\Big|\int
e^{i\lambda^2}e^{-
i\bar{r}\lambda}\lambda^{n-1}a(t^{-1/2}\lambda,z,z')
\left(1-\phi_0(\lambda)\right) \big( 1 - \phi_0(4\bar{r} \lambda)
\big) d\lambda\Big|.
\end{split}
\end{equation*}
Let $\Phi(\lambda,\bar{r})=\lambda^2-\bar{r}\lambda$. We first
consider $II_1$. Since the integral for $II_1$ is supported where $\lambda \leq (4 \bar{r})^{-1}$ and $\lambda \geq
1/2$, the integrand is only
nonzero when $\bar{r}\leq1/2$. Therefore $|\partial_\lambda\Phi| =
2\lambda - \bar{r} \geq\frac12\lambda$. Define the operator
$L=L(\lambda,\bar{r})=(2\lambda-\bar{r})^{-1}\partial_\lambda$. By
\eqref{beans0} and using integration by parts, we obtain for $N>n/2$
\begin{equation*}
\begin{split}
II_1\leq&\sum_{m\geq0}\Big|\int
e^{i\lambda^2}e^{-
i\bar{r}\lambda}\lambda^{n-1}a(t^{-1/2}\lambda,z,z')\phi(\frac{\lambda}{2^m})\phi_0(4\bar{r} \lambda) d\lambda\Big|
\\=&\sum_{m\geq0}\Big|\int L^{N}
\big(e^{i(\lambda^2-
\bar{r}\lambda)}\big)\Big[\lambda^{n-1}a(t^{-1/2}\lambda,z,z')\phi(\frac{\lambda}{2^m})\phi_0(4\bar{r}
\lambda) \Big]d\lambda\Big|
\\\leq &C_N\sum_{m\geq0}\int_{|\lambda|\sim
2^{m}}\lambda^{n-1-2N}d\lambda\leq C_N.
\end{split}
\end{equation*}
Finally we consider $II_2$. Here, we replace the decomposition $\sum_m \phi(2^{-m} \lambda)$ with a different decomposition,
based on the size of $\partial_\lambda \Phi$.
\begin{equation*}
\begin{split}
II_2\leq& \Big|\int e^{i\lambda^2}e^{-
i\bar{r}\lambda}\lambda^{n-1}a(t^{-1/2}\lambda,z,z')\big(1-\phi_0(\lambda)\big)\phi_0(2\lambda-\bar{r}) \big( 1 - \phi_0(4\bar{r} \lambda) \big) \, d\lambda\Big|\\
&+\sum_{m\geq0}\Big|\int
e^{i\lambda^2}e^{-
i\bar{r}\lambda}\lambda^{n-1}a(t^{-1/2}\lambda,z,z')
\big(1-\phi_0(\lambda)\big)\phi(\frac{2\lambda-\bar{r}}{2^m})\big( 1 - \phi_0(4\bar{r} \lambda) \big) \, d\lambda\Big|\\:=&II_2^1+II_2^2.
\end{split}
\end{equation*}
If $\bar{r} \leq 10$, then for the integrand of $II_2^1$ to be
nonzero we must have $\lambda \leq 10$, due to the $\phi_0$ factor.
Then it is easy to see that $II_2^1$ is uniformly bounded. If $\bar{r} \geq
10$, we have $\bar{r}\sim\lambda$ since $|2\lambda-\bar{r}|\leq 1$.
Hence, using \eqref{beans0} with $\alpha = 0$,
$$II_2^1\le\int_{|2\lambda-\bar{r}|\leq 1}\lambda^{n-1}(1+\bar{r}\lambda)^{-\frac{n-1}2}d\lambda\leq C.$$
Now we consider the second term. Integrating by parts, we show by
\eqref{beans0}
\begin{equation*}
\begin{split}
II_2^2\leq&\sum_{m\geq0}\Big|\int
e^{i\lambda^2}e^{-
i\bar{r}\lambda}\lambda^{n-1}a(t^{-1/2}\lambda,z,z')
\big(1-\phi_0(\lambda)\big)\phi(\frac{2\lambda-\bar{r}}{2^m})\big( 1 - \phi_0(4\bar{r} \lambda) \big) \, d\lambda\Big|
\\=&\sum_{m\geq0}\Big|\int L^{N}
\big(e^{i(\lambda^2-
\bar{r}\lambda)}\big)\Big[\lambda^{n-1}a(t^{-1/2}\lambda,z,z')\big(1-\phi_0(\lambda)\big)\phi(\frac{2\lambda-\bar{r}}{2^m}) \big( 1 - \phi_0(4\bar{r} \lambda) \big) \Big] \, d\lambda\Big|
\\ \leq &C_N\sum_{m\geq0}2^{-mN}\int_{|2\lambda-\bar{r}|\sim 2^m}\lambda^{n-1}(1+\bar{r}\lambda)^{-\frac{n-1}2}d\lambda .
\end{split}
\end{equation*}
If $\bar{r}\leq 2^{m+1}$, then $\lambda \leq 2^{m+2}$ on the support
of the integrand. The $m$th term can then be estimated by $C_N
2^{-mN} 2^{(m+2)n}$ which is summable for $N > n$.  Otherwise, we
have $\lambda\sim \bar{r}$, which means the integrand is bounded and
we estimate the $m$th term by $C_N 2^{-mN} 2^m$, which is summable
for $N > 1$.  Therefore we have completed the proof of Proposition
\ref{dispersive}.\end{proof}\vspace{0.2cm}


\section{Homogeneous Strichartz estimates}\label{sec:hom}  We use the $L^2$-estimates and the microlocalized dispersive estimates to conclude the proof
of Theorem \ref{Strichartz}. By Proposition \ref{prop:L2}, we have
for all $t\in\mathbb{R}$ and all $u_0\in L^2$
\begin{equation*}
\|U_j(t)u_0\|_{L^2(M^\circ)}\lesssim \|u_0\|_{L^2(M^\circ)};
\end{equation*}
By Lemma~\ref{BB*},
\begin{equation*} U_j(s)U_j^*(t)f=\int_0^\infty e^{i(s-t)\lambda^2}
Q_j(\lambda)dE_{\sqrt{\mathbf{H}}}(\lambda)Q^*_j(\lambda)f.
\end{equation*}
Hence we have the following decay estimates by Proposition
\ref{dispersive}
\begin{equation*}
\|U_j(s)U_j^*(t)f\|_{L^\infty}\lesssim |t-s|^{-n/2}\|f\|_{L^1}.
\end{equation*}
As a consequence of the Keel-Tao abstract Strichartz estimate in
\cite{KT}, we have 
\begin{equation}
\|U_j(t)u_0\|_{L^q(\mathbb{R}; L^r(M^\circ))}\lesssim
\|u_0\|_{L^2(M^\circ)},
\label{micro-Strichartz}\end{equation}
where $(q,r)$ is sharp $\frac n2$-admissible, that is, $q,r\geq2$,
$(q,r,n)\neq(2,\infty,2)$ and $2/q+n/r=n/2$. By the definition of
$U_j(t)$ based on the construction of $Q_j$, we see that
\begin{equation}
e^{it\mathbf{H}}=\sum_{j=1}^N U_j(t).
\label{U_jsum}\end{equation}
Combining \eqref{micro-Strichartz} and \eqref{U_jsum} proves the long-time homogeneous Strichartz estimate.


\section{Inhomogeneous Strichartz estimates}\label{sec:inhom}
In this section, we prove Theorem~\ref{Strichartz-inhom},
including at the endpoint $(q,r)= (\tilde q, \tilde r) = (2,\frac{2n}{n-2})$ for $n \geq 3$. Let
$\mathbf{U}(t)=e^{it\mathbf{H}}: L^2\rightarrow L^2$. We have already proved that
\begin{equation*}
\|\mathbf{U}(t)u_0\|_{L^q_tL^r_z}\lesssim\|u_0\|_{L^2}
\end{equation*} holds for all $(q,r)$ satisfying \eqref{1.1}.
By duality, the estimate is equivalent to
\begin{equation*}
\Big\|\int_{\R}\mathbf{U}(t)\mathbf{U}^*(s)F(s)ds\Big\|_{L^q_tL^r_z}\lesssim\|F\|_{L^{\tilde{q}'}_tL^{\tilde{r}'}_z},
\end{equation*}
where both $(q,r)$ and $(\tilde{q},\tilde{r})$ satisfy \eqref{1.1}.
By the Christ-Kiselev lemma \cite{CK}, we obtain for $q>\tilde{q}'$
\begin{equation}
\Big\|\int_{s<t}\mathbf{U}(t)\mathbf{U}^*(s)F(s)ds\Big\|_{L^q_tL^r_z}\lesssim\|F\|_{L^{\tilde{q}'}_tL^{\tilde{r}'}_z}.
\end{equation}
Notice that $\tilde{q}'\leq 2 \leq q$, therefore we have proved all inhomogeneous
Strichartz estimates except the endpoint
$(q,r)=(\tilde{q},\tilde{r})=(2,\frac{2n}{n-2})$. To treat the
endpoint, we need show the bilinear form estimate
\begin{equation}\label{TFG}
|T(F,G)|\leq \|F\|_{L^2_tL^{r'}_z}\|G\|_{L^2_tL^{r'}_z},
\end{equation}
where $r=2n/(n-2)$ and $T(F,G)$ is the bilinear form
\begin{equation}
T(F,G)=\iint_{s<t}\langle \mathbf{U}(t)\mathbf{U}^*(s)F(s), G(t)\rangle_{L^2}~ dsdt.
\end{equation}

Theorem~\ref{Strichartz-inhom} follows from

\begin{proposition}\label{either} There exists a partition of the identity $Q_j(\lambda)$ on $L^2(M^\circ)$  such that, with $U_j(t)$ defined as in \eqref{Uit},
there exists a constant $C$ such that for each pair $(j,k)$,  either
\begin{equation}\label{bilinear:s<t}
\iint_{s<t}\langle U_j(t)U_k^*(s)F(s), G(t)\rangle_{L^2}~ dsdt\leq C
\|F\|_{L^2_tL^{r'}_z}\|G\|_{L^2_tL^{r'}_z}.
\end{equation}
or
\begin{equation}\label{bilinear:s>t}
\iint_{s>t}\langle U_j(t)U_k^*(s)F(s), G(t)\rangle_{L^2}~ dsdt\leq C
\|F\|_{L^2_tL^{r'}_z}\|G\|_{L^2_tL^{r'}_z}.
\end{equation}
\end{proposition}

\begin{proof}[Proof of Theorem~\ref{Strichartz-inhom} assuming Proposition~\ref{either}]
We have proved that for all $1\leq j\leq N$,
\begin{equation*}
\|U_j(t)u_0\|_{L^2_tL^r_z}\lesssim\|u_0\|_{L^2},
\end{equation*}
hence it follows by duality that for all $1\leq j,k\leq N$,
\begin{equation}\label{R2}
\iint_{\R^2}\langle U_j(t)U_k^*(s)F(s), G(t)\rangle_{L^2}~ dsdt\leq
C \|F\|_{L^2_tL^{r'}_z}\|G\|_{L^2_tL^{r'}_z}.
\end{equation}
Subtracting \eqref{bilinear:s>t} from \eqref{R2}  shows that
\eqref{bilinear:s<t} holds for every pair $(j,k)$.  Then, by summing
over all $j$ and $k$, we obtain \eqref{TFG}.
\end{proof}

To prove Proposition~\ref{either} we use the following lemma proved in \cite[Lemmas 5.3 and 5.4]{GH}.

\begin{lemma}\label{poi}
The partition of the identity $Q_j(\lambda)$ can be chosen so that
the pairs of indices $(j,k)$, $1 \leq j,k \leq N$, can be divided
into three classes,
$$
\{ 1, \dots, N \}^2 = J_{near} \cup J_{not-out} \cup J_{not-inc},
$$
so that
\begin{itemize}
\item if $(j,k) \in J_{near}$, then $Q_j(\lambda) \specl Q_k(\lambda)^*$ satisfies the conclusions of Proposition~\ref{prop:localized spectral measure};

\item if $(j,k) \in J_{non-inc}$, then $Q_j(\lambda)$ is not incoming-related to $Q_k(\lambda)$ in the sense that no point in the operator wavefront set (microlocal support)
of $Q_j(\lambda)$ is related to a point in the operator wavefront
set of $Q_k(\lambda)$ by backward bicharacteristic flow;

\item if $(j,k) \in J_{non-out}$, then $Q_j(\lambda)$ is not outgoing-related to $Q_k(\lambda)$ in the sense that no point in the operator wavefront set  of
$Q_j(\lambda)$ is related to a point in the operator wavefront set
of $Q_k(\lambda)$ by forward bicharacteristic flow.
\end{itemize}
\end{lemma}

We exploit the not-incoming or not-outgoing property of
$Q_j(\lambda)$ with respect to $Q_k(\lambda)$ in the following two
lemmas.

\begin{lemma}\label{lem:sign-low} Let $Q_j(\lambda), Q_k(\lambda)$ be such that $Q_j$ is not outgoing-related to $Q_k$.  Then, for $\lambda \leq 2$, and as a multiple of $|dg dg'|^{1/2} |d\lambda|$,
the Schwartz kernel of $Q_j(\lambda) \specl Q_k(\lambda)^*$
can be expressed  as the sum of a finite number of terms of the form
\begin{gather}
\lambda^{n-1}  \int_{\R^k} e^{i\lambda\Phi(y,y',\sigma,v)/x}  \big(\frac{x'}{\lambda}\big)^{(n-1)/2 - k/2}  a(\lambda,y,y',\sigma,\frac{x'}{\lambda},v)dv \quad \mathrm{ or } \label{QiEQj-lo}\\
\lambda^{n-1}  \int_{\R^{k-1}} \int_0^\infty e^{i\lambda\Phi(y,y',\sigma,v,s)/x} \big(\frac{x'}{\lambda s}\big)^{(n-1)/2 - k/2}
  s^{n-2} a(\lambda,y,y',\sigma,\frac{x'}{\lambda},v,s) \, ds \, dv  \label{QiEQj-s-lo}
  \end{gather}
  in the region $\sigma = x/x' \leq 2$, $x'/\lambda \leq 2$, or
  \begin{gather}
\quad  \lambda^{n-1} a(\lambda,y,y',\sigma,x'/\lambda) \label{QiEQj-c-lo}
\end{gather}
in the region $\sigma = x/x' \leq 2$, $x'/\lambda \geq 1$,
where in each case, $\Phi < - \epsilon < 0$ and $a$ is a smooth function compactly supported in the $v$ and $s$ variables (where present), such that $|(\lambda\partial_\lambda)^N a|\leq C_N$ for all $N \in \NN$.
In each case, we may assume that $ k \leq n-1$; if $k=0$ in \eqref{QiEQj-lo}  or $k=1$ in \eqref{QiEQj-s-lo} then there is no variable $v$, and no $v$-integral. The key point is that in each expression, \emph{the phase function is strictly negative}.

If, instead, $Q_j$ is not incoming-related to $Q_k$, then the same
conclusion holds with the reversed sign: the Schwartz kernel can be
written as a finite sum of terms with a strictly positive phase
function.
\end{lemma}

\begin{remark}
For $\sigma \geq 1/2$, the Schwartz kernel has a similar description, as follows immediately from the symmetry of the kernel under interchanging the left and right variables.
\end{remark}

\begin{proof}
The statement that the Schwartz kernel has the indicated forms above
follows immediately from the description of the spectral measure in
\cite[Theorem 3.10]{GHS1} as a Legendre distribution in the class
$I^{m, p; r_{\lb}, r_{\rb}}(\MMkb, (L^{\bfc}, L^\sharp);
\Omegakbh)$, where $m = -1/2$, $p = (n-2)/2$, $r_{\lb} = r_{\rb} =
(n-1)/2$. The bound on $k$ follows from the fact that $k$ can be
taken as the drop in rank of the projection from $L^{\bfc}$ to the
base $(\partial M)^2 \times (0, \infty)_\sigma$ which is the front
face (that is, the face created by blowup) of $M^2_b$. We claim that
the drop in rank is at most $n-1$, which proves that we may assume
that $k \leq n-1$. To prove this claim, we show that the
differentials $dy_1, \dots dy_{n-1}$ and at least one of $d\sigma,
dy'_1, \dots, dy'_{n-1}$ are linearly independent on $L^{\bf}$. This
can be seen from the description of $L^{\bf}$ as the flowout from
the set
\begin{equation}
\{ (y, y, 1, \mu, -\mu, \nu, -\mu) \mid \nu^2 + h(\mu) = 1 \},
\label{initial-surface}\end{equation}
using the coordinates of \eqref{gamma^2},
by the flow of the vector field $V_r$,  which is the vector field given by $x^{-1}$ times the Hamilton vector field of the principal symbol of $\Delta$ acting in the right variables on $\MMkb$. In fact $V_r = \sin s' \partial_{s'}$ in the coordinates $(s,s')$ on the leaves $\gamma^2$ of \eqref{gamma^2}, and takes the form (see \cite[Eq. (2.26)]{HV2} or \cite[Eq. (3.5)]{GHS1})
$$
2\nu' \sigma \frac{\partial}{\partial\sigma} - 2\nu' \mu' \cdot \frac{\partial}{\partial \mu'} + h' \frac{\partial}{\partial \nu'} + \big( \frac{\partial h'}{\partial \mu'} \frac{\partial}{\partial y'} - \frac{\partial h'}{\partial y'} \frac{\partial}{\partial \mu'} \Big), \quad h' = h(y', \mu') = \sum_{i,j} h^{ij}(y') \mu'_i \mu'_j.
$$
It is clear that $dy_1, \dots, dy_{n-1}$ are linearly independent at the initial set \eqref{initial-surface}. Moreover their Lie derivative with respect to $V_r$ vanishes, so they are linearly dependent on all of $L^{\bfc}$. Also, since $h' + {\nu'}^2 = 1$ on $L^{\bfc}$, either the $\partial_\sigma$ or the $\partial_{y'}$ component of the vector field $V_r$ does not vanish, unless $\sigma = 0$, showing that either $d\sigma$ or one of the $dy'_i$ do not vanish at each point of $L^{\bfc}$ for $\sigma \neq 0$. But it was shown in \cite{HV2} that $L^{\bfc}$ is transversal to the boundary at $\sigma = 0$, which means that $d\sigma \neq 0$ on $L^{\bfc}$ when $\sigma$ is small. This proves the claim.

We next show that $\Phi$ can be taken to be strictly negative. We
use the microlocal support estimates from \cite{GHS2}. Applying
\cite[Corollary 5.3]{GHS2}, we find that the microlocal support of
$Q_j(\lambda) \specl Q_k(\lambda)^*$ is contained in that part of
$L^{\bfc}$ where (in the notation of \eqref{gamma^2}) $s < s'$
(since the initial set \eqref{initial-surface} corresponds to
$s=s'$, and $\partial_s$, respectively $\partial_{s'}$ moves in the
outgoing, resp. incoming, direction along the flow). Repeating the
calculation following \eqref{gamma^2} we see that the value of
$\Phi$ `on the Legendrian' is $\Phi = -\cos s + \sigma \cos s' =
(\sin s')^{-1} \sin(s-s')$, which is strictly negative. By
restricting the support of the amplitude $a$ in \eqref{QiEQj-lo} ---
\eqref{QiEQj-c-lo}, we can assume that $\Phi$ is negative everywhere
on the support of the integrand.
\end{proof}

\begin{lemma}\label{lem:sign} Let $Q_j(\lambda), Q_k(\lambda)$ be such that $Q_j$ is not outgoing-related to $Q_k$.  Then, for $\lambda \geq 1$,
and as a multiple of $|dg dg'|^{1/2} |d\lambda|$, the Schwartz
kernel of $Q_j(\lambda) \specl Q_k(\lambda)^*$
can be written in terms of a finite number of oscillatory integrals
of the form
\begin{gather}
 \int_{\R^k} e^{i\lambda\Phi(y,y',\sigma,x,v)/x}\lambda^{n-1+k/2} x^{(n-1)/2-k/2}a(\lambda,y,y',\sigma,x,v)dv \quad \mathrm{ or } \label{QiEQj-hi}\\
  \int_{\R^{k-1}} \int_0^\infty e^{i\lambda\Phi(y,y',\sigma,x,v,s)/x}\lambda^{n-1+k/2} \big(\frac{x}{s}\big)^{(n-1)/2 - k/2}
  s^{n-2} a(\lambda,y,y',\sigma,x,v,s) \, ds \, dv  \label{QiEQj-s}
  \end{gather}
in the region $\sigma = x/x' \leq 2$, $x \leq \delta$, or
\begin{gather}
  \int_{\RR^k} e^{i\lambda \Phi(z, z', v)} \lambda^{n-1+k/2} a(\lambda, z, z', v) \, dv \label{QiEQj}
\end{gather}
in the region $x \geq \delta, x' \geq \delta$,
where in each case, $\Phi < - \epsilon < 0$ and $a$ is a smooth function compactly supported in the $v$ and $s$ variables (where present), such that $|(\lambda\partial_\lambda)^N a|\leq C_N$.
In each case, we may assume that $k \leq n-1$; if $k=0$ in \eqref{QiEQj-hi} or \eqref{QiEQj}, or $k=1$ in \eqref{QiEQj-s} then there is no variable $v$, and no $v$-integral. Again, the key point is that in each expression, \emph{the phase function is strictly negative}.

If, instead, $Q_j$ is not incoming-related to $Q_k$, then the same
conclusion holds with the reversed sign: the Schwartz kernel can be
written as a finite sum of terms with a strictly positive phase
function.
\end{lemma}

\begin{proof}
The proof is essentially identical to that of
Lemma~\ref{lem:sign-low}. The form of the oscillatory integrals
comes from the fact that the spectral measure, for high energies, is
a Legendre distribution in the class $I^{m,  p; r_{\lb}, r_{\rb}}(X,
(L, L^\sharp); \Omega {}^{s\Phi} \Omega^{1/2})$, where the
Lagrangian $L$ is given by  \eqref{L}. The non-outgoing relation
implies, via the microlocal support estimates of \cite[Section
7]{GHS2} that $Q_j(\lambda) dE_{\sqrt{\mathbf{H}}}(\lambda)
Q_k(\lambda)^*$ is microsupported where $\tau < 0$ in the
coordinates of \eqref{L}. Since $\Phi = \tau$ when $d_v \Phi = 0$,
this implies that $\Phi < 0$ when $d_v \Phi = 0$. By restricting the
support of the amplitude close to the set where $d_v \Phi = 0$, we
can assume that $\Phi < 0$ everywhere on the support of the
integrand.
\end{proof}

Next we establish dispersive estimates for $U_j(t) U_k(s)^*$:

\begin{lemma}\label{lem:Qij} We have the following estimates on $U_j(t) U_k(s)^*$:

\begin{itemize}
\item If $(j,k) \in J_{near}$, then for all $t \neq s$ we have
\begin{equation}
\big\|U_j(t)U^*_k(s)\big\|_{L^1\rightarrow L^\infty}\leq C
|t-s|^{-\frac{n}2}, \label{UiUjnear}\end{equation}

\item If $(j,k)$ such that $Q_j$ is not outgoing related to
$Q_k$, and $t<s$, then
\begin{equation}
\big\|U_j(t)U^*_k(s)\big\|_{L^1\rightarrow L^\infty}\leq C
|t-s|^{-\frac{n}2}, \label{UiUj}\end{equation}

\item Similarly, if $(j,k)$ such that $Q_j$ is not incoming
related to $Q_k$, and $s<t$, then
\begin{equation}
\big\|U_j(t)U^*_k(s)\big\|_{L^1\rightarrow L^\infty}\leq C
|t-s|^{-\frac{n}2}. \label{UiUj2}\end{equation}
\end{itemize}
\end{lemma}

\begin{proof} The estimate \eqref{UiUjnear} is essentially proved in Proposition \ref{dispersive}, since we can use
Proposition \ref{prop:localized spectral measure}. Assume that $Q_j$
is not incoming-related to $Q_k$, and consider \eqref{UiUj2}.
 By Lemma~\ref{BB*}, $U_j(t) U_k(s)^*$ is given by
\begin{equation}
\int_0^\infty e^{i(t-s)\lambda^2} \big(Q_j(\lambda)
dE_{\sqrt{\mathbf{H}}}(\lambda)Q^*_k(\lambda)\big)(z,z').
\label{UiUjint}\end{equation} Then we need to show that for $s<t$
\begin{equation}
\Big|\int_0^\infty e^{i(t-s)\lambda^2} \big(Q_j(\lambda)
dE_{\sqrt{\mathbf{H}}}(\lambda)Q^*_k(\lambda)\big)(z,z')
d\lambda\Big|\leq C |t-s|^{-\frac{n}2}.
\end{equation}

\textbf{Case 1, $t-s \geq 1$.} We introduce a dyadic partition of unity in $\lambda$. Let $\phi\in C_c^\infty([\frac12,2])$ be   as in Section~\ref{sec:L2}, such that
$\sum_m\phi(2^{-m}\sqrt{t-s}\lambda)=1$, define
$$\phi_0(\sqrt{t-s}\lambda)=\sum_{m\leq0}\phi(2^{-m}\sqrt{t-s}\lambda),$$ and insert
$$
1 = \phi_0(\sqrt{t-s}\lambda) + \sum_{m \geq 1} \phi_m(\sqrt{t-s}\lambda), \quad \phi_m(\lambda) :=  \phi(2^{-m}\lambda)
$$
 into the integral  \eqref{UiUjint}. In addition, we substitute for  $Q_j(\lambda)
dE_{\sqrt{\mathbf{H}}}(\lambda)Q^*_k(\lambda)$ one of the
expressions in Lemmas~\ref{lem:sign-low} and \ref{lem:sign}. Since
$t-s \geq 1$, for the $\phi_0$ term, only the low energy expressions
are relevant. The estimate  follows immediately from noticing that
these expressions are pointwise bounded by $C \lambda^{n-1}$, using
the fact that $k \leq n-1$ in these expressions.

To treat the $\phi_m$ terms for $m \geq 1$, we substitute again one of the expressions in Lemmas~\ref{lem:sign-low} and \ref{lem:sign}.
For notational simplicity we consider the expression  \eqref{QiEQj}, but the argument is similar in the other cases.   We scale the $\lambda$ variable and obtain the expression
\begin{equation}\begin{gathered}
\int_0^\infty \int_{\R^{k}}e^{i(t-s)\lambda^2}
e^{i\lambda\Phi(z,z',v)}\lambda^{n-1+k/2}a(\lambda,z,z',v)\phi_m(
\sqrt{t-s} \lambda) \, dv \, d\lambda
\\
= (t-s)^{-\frac n2 - \frac{k}{4}}\int_0^\infty \int_{\R^{k}}
e^{i\big(\lambdabar^2+\frac{\lambdabar\Phi(z,z',v)}{\sqrt{t-s}}\big)}\lambdabar^{n-1+k/2}
a(\frac{\lambdabar}{\sqrt{t-s}},y,y',\sigma,v) \phi_m(\lambdabar) \,
dv \, d\lambdabar \end{gathered} \label{intt}
\end{equation}
where $\lambdabar = \sqrt{t-s} \lambda$.
We observe that the overall exponential factor is invariant under the differential operator
$$
L = \frac{-i}{2\lambdabar^2 + \lambdabar \Phi/\sqrt{t-s}} \lambdabar
\frac{\partial}{\partial \lambdabar}.
$$
The adjoint of this is
$$
L^t = -L + \frac{i}{2\lambdabar^2 + \lambdabar \Phi/\sqrt{t-s}} -
i\frac{4\lambdabar^2 + \lambdabar \Phi/\sqrt{t-s}}{(2\lambdabar^2 + \lambdabar
\Phi/\sqrt{t-s})^2}.
$$
We apply $L^N$  to the exponential factors, and integrate by
parts $N$ times. Since $\Phi \geq 0$ according to Lemma~\ref{lem:sign}, and
since we have an estimate $|(\lambdabar \partial_\lambdabar)^N a| \leq
C_N$, each time we integrate by parts we gain a factor $\lambdabar^{-2}
\sim 2^{-2m}$. It follows that  the integral with
$\phi(2^{-m}\lambdabar)$ inserted  is bounded by $(t-s)^{-n/2} 2^{-m(2N-n - k/2)}$
uniformly for $t-s \geq 1$. Hence we prove \eqref{UiUj2} by summing over
$m\geq0$. The argument to prove \eqref{UiUj} is analogous.

\textbf{Case 2, $t-s \leq 1$.} In this case, we use a dyadic decomposition in terms of the original variable $\lambda$. We consider the integral \eqref{UiUjint}, insert the dyadic decomposition
$$
1 = \sum_{m\geq 0} \phi_m(\lambda),
$$
and substitute for  $Q_j(\lambda)
dE_{\sqrt{\mathbf{H}}}(\lambda)Q^*_k(\lambda)$ one of the
expressions in Lemmas~\ref{lem:sign-low} and \ref{lem:sign}.

For the case $m = 0$, the estimate follows immediately from the uniform boundedness of \eqref{QiEQj-lo} --- \eqref{QiEQj-c-lo}.
For the cases $m \geq 1$, we use the expressions in Lemma~\ref{lem:sign} and observe that the overall exponential factor is invariant under the differential operator
$$
L = \frac{-i}{2(t-s)\lambda^2 + \lambda \Phi} \lambda
\frac{\partial}{\partial \lambda}.
$$
The adjoint of this is
$$
L^t = -L + \frac{i}{2(t-s)\lambda^2 + \lambda \Phi} -
i\frac{4(t-s)\lambda^2 + \lambda \Phi}{(2(t-s)\lambda^2 + \lambda
\Phi)^2}.
$$
We apply $L$ $N$-times to the exponential factors, and integrate by
parts. Since $\Phi \geq \epsilon > 0$ according to Lemma~\ref{lem:sign}, and
since we have an estimate $|(\lambda \partial_\lambda)^N a| \leq
C_N$, each time we integrate by parts we gain a factor $\lambda^{-1}
\sim 2^{-m}$. It follows that  the integral with
$\phi(2^{-m}\lambda)$ inserted  is bounded by $ 2^{-m(N-n - k/2)}$
uniformly for $t-s \leq 1$. Hence we prove \eqref{UiUj2} by summing over
$m\geq0$. The argument to prove \eqref{UiUj} is analogous.
\end{proof}

\begin{remark} Notice that, in the cases \eqref{UiUj} and \eqref{UiUj2}, there is a lot of `slack' in the estimates. This is because the sign of $t-s$ has the favourable sign relative to the sign of the phase function, so that the overall phase in integrals such as  \eqref{intt} are never stationary. Then integration by parts give us more decay than needed to prove the estimates. This is important because it overcomes the growth of the spectral measure as $\lambda \to \infty$ at conjugate points: at pairs of conjugate points we have $k > 0$ and we see from, say,  \eqref{QiEQj} that the spectral measure will not obey the localized (near the diagonal) estimates of Proposition~\ref{prop:localized spectral measure}, by a factor $\lambda^{k/2}$. The geometric meaning of $k$ is  the  drop in rank of the projection from $L$ down to $M^2_b$, hence is positive precisely at pairs of conjugate points.
\end{remark}

We now complete the proof of Theorem~\ref{Strichartz-inhom} by proving Proposition~\ref{either}.
\begin{proof}[Proof of Proposition~\ref{either}] We use a partition of the identity as in Lemma~\ref{poi}. In the case that $(j,k) \in J_{near}$,
we have the dispersive estimate \eqref{UiUjnear}. This allows us to
apply the argument of \cite[Sections 4--7]{KT} to obtain
\eqref{bilinear:s<t}. In the case that $(j,k) \in J_{non-out}$, we
obtain \eqref{bilinear:s<t} following the argument in \cite{KT}
since we have the dispersive estimate \eqref{UiUj2} when $s < t$.
Finally, in the case that $(j,k) \in J_{non-inc}$, we obtain
\eqref{bilinear:s>t} since we have the dispersive estimate
\eqref{UiUj} for $s > t$.
\end{proof}

\begin{remark} The endpoint inhomogeneous Strichartz estimate
is closely related to the  uniform Sobolev estimate
\begin{equation}
\| (\mathbf{H} - \alpha)^{-1} \|_{L^r \to L^{r'}} \leq C, \quad r = \frac{2n}{n+2},
\label{unifSob}\end{equation}
where  $C$ is independent of $\alpha \in \CC$.
This estimate was proved by \cite{KRS} for the flat Laplacian, and
by \cite{GH} for the Laplacian on nontrapping asymptotically conic
manifolds (it was also shown in \cite{GH} that \eqref{unifSob} holds for $r \in [2n/(n+2), 2(n+1)/(n+3)]$ with a power of $\alpha$ on the RHS). In fact, it was pointed out to the authors by Thomas Duyckaerts and Colin Guillarmou that the endpoint inhomogeneous Strichartz estimate implies the uniform Sobolev estimate \eqref{unifSob}.
To see this, we choose $w \in C_c^\infty(M^\circ)$ and
$\chi(t)$ equal to $1$ on $[-T, T]$ and zero for $|t| \geq T+1$,
  and let $u(t, z) = \chi(t) e^{i\alpha t} w(z)$.
Then
$$
(i\partial_t + \mathbf{H}) u = F(t, z), \quad F(t, z) := \chi(t)
e^{i\alpha t} (\mathbf{H} - \alpha) w(z) + i \chi'(t) e^{i\alpha t}
w(z).
$$
Applying the endpoint inhomogeneous Strichartz estimate, we obtain
$$
\| u \|_{L^2_t L^{r'}_z} \leq C \| F \|_{L^2_t L^r_z}.
$$
From the specific form of $u$ and $F$ we have
$$
\| u \|_{L^2_t L^{r'}_z} = \sqrt{2T} \| w \|_{L^{r'}} + O(1), \quad
\|F \|_{L^2_t L^r_z} = \sqrt{2T} \| (\mathbf{H} - \alpha) w \|_{L^{r}} + O(1).
$$
Taking the limit $T \to \infty$ we find that
$$
\| w \|_{L^{r'}} \leq C \| (\mathbf{H} - \alpha) w \|_{L^{r}},
$$
which implies the uniform Sobolev estimate.

In the other direction, suppose that the uniform Sobolev estimate
holds. If $u$ and $F$ satisfy \eqref{uF}, then taking the Fourier transform in $t$ we find that
\begin{equation}
(\mathbf{H} - \alpha) \hat u(\alpha, z) = \hat F(\alpha, z).
\label{ft}\end{equation}
Suppose for a moment that the following statement were true: ``Fourier transformation in $t$ is a bounded linear map from $L^2(\RR_t; L^{p}(M^\circ))$ to $L^2(\RR_\alpha; L^p(M^\circ))$ for $p = r', r$''.  Using this and the uniform Sobolev inequality, applied to \eqref{ft}, we would obtain the inhomogeneous Strichartz estimate. Unfortunately, the statement in quotation marks is known to be false, so this argument is purely heuristic. Nevertheless, it illustrates the close relation between the two estimates. It would be interesting to know if there are general conditions under which the two estimates are equivalent.
\end{remark}

\begin{center}

\end{center}

\end{document}